\documentclass[regno]{amsart}
\usepackage[utf8]{inputenc}
\def\bigskip{\vspace{14pt}}

\usepackage[utf8]{inputenc} 
\usepackage{amssymb}
\usepackage{mathrsfs}
\usepackage{graphicx}
\usepackage{latexsym}
\usepackage{stmaryrd}
\usepackage{enumitem}
\usepackage[bookmarks,hyperfootnotes=false]{hyperref}
\usepackage{multirow}
\usepackage{xcolor}
\usepackage{caption}
\colorlet{darkishRed}{red!80!black}
\colorlet{darkishBlue}{blue!60!black}
\colorlet{darkishGreen}{green!60!black}
\hypersetup{
    draft = false,
    bookmarksopen=true,
    colorlinks,
    linkcolor={red!60!black},
    citecolor={green!60!black},
    urlcolor={blue!60!black}
}
\usepackage[nameinlink, capitalise, noabbrev]{cleveref}
\crefformat{enumi}{#2#1#3}
\crefformat{equation}{#2(#1)#3}
\usepackage{doi}

\let\setminus=\smallsetminus
\usepackage{tikz}
\usepackage{tikz-cd}
\usetikzlibrary{calc,through,intersections,arrows, trees, positioning, decorations.pathmorphing, cd}
\usepackage{relsize}
\usepackage{comment}
\usepackage{svg}
\usepackage{mathtools}
\usepackage{nccmath}
\usepackage{pifont}
\usepackage{comment}

\usepackage[skip=6pt]{subcaption} 
\usepackage{geometry}
\let\setminus=\smallsetminus
\renewcommand{\leq}{\leqslant}
\renewcommand{\geq}{\geqslant}

\def\namedlabel#1#2{\begingroup
   \def\@currentlabel{#2}%
   \label{#1}\endgroup
}

\let\rho=\varrho
\let\phi=\varphi

\DeclareMathOperator{\interior}{int}

\newcommand{\shiftingfct}[2]{f\mkern-6mu\downarrow_{#1}^{#2}}

\newcommand{ \R } { \mathbb{R} }

\newcommand{ \N } { \mathbb{N} }

\newcommand{\orange}[1]{\textcolor{orange}{#1}}

\newcommand{\abs}[1]{\lvert#1\rvert}

\makeatletter

\def\calCommandfactory#1{%
   \expandafter\def\csname c#1\endcsname{\mathcal{#1}}}
\def\frakCommandfactory#1{%
   \expandafter\def\csname frak#1\endcsname{\mathfrak{#1}}}

\newcounter{ctr}
\loop
  \stepcounter{ctr}
  \edef\X{\@Alph\c@ctr}
  \expandafter\calCommandfactory\X
  \expandafter\frakCommandfactory\X
  \edef\Y{\@alph\c@ctr}
  \expandafter\frakCommandfactory\Y
\ifnum\thectr<26
\repeat

\setenumerate{label={\normalfont (\roman*)}}

\usepackage[presets={vec-cev,abc,ABC,cAcBcC}]{letterswitharrows}

\newtheorem{theorem}{Theorem}[section] 
\newtheorem{proposition}[theorem]{Proposition}

\newtheorem{lemma}[theorem]{Lemma}

\newtheorem{LEM}[theorem]{Lemma}
\newtheorem{THM}[theorem]{Theorem}
\newtheorem{COR}[theorem]{Corollary}

\theoremstyle{definition}

\newtheorem{claim}{Claim}
\crefname{claim}{Claim}{Claims}

\newenvironment{claimproof}{\noindent\textit{Proof.}}{\hfill\ensuremath{\blacksquare}\medskip}
\usepackage{etoolbox}

\newlist{thmlist}{enumerate}{1}
\setlist[thmlist]{label=(\roman{thmlisti}), ref=\thethm.(\roman{thmlisti}),noitemsep}

\def\td{tree-decom\-po\-sition}
\def\tds{tree-decom\-po\-sitions}

\newtheorem{innercustomthm}{Theorem}

\newenvironment{customthm}[1]
  {\renewcommand\theinnercustomthm{#1}\innercustomthm}
  {\endinnercustomthm}

\theoremstyle{definition}
\newtheorem{example}[theorem]{Example}

\theoremstyle{remark}

\title{Optimal trees of tangles: refining the essential parts}
\author{Sandra Albrechtsen}


\begin{document}
	\maketitle

\pagenumbering{arabic}

\begin{abstract}
    \noindent 
    We combine the two fundamental fixed-order tangle theorems of Robertson and Seymour into a single theorem that implies both, in a best possible way.
    We show that, for every $k \in \N$, every \td\ of a graph $G$ which efficiently distinguishes all its $k$-tangles can be refined to a \td\ whose parts are either too small to be home to a $k$-tangle, or as small as possible while being home to a $k$-tangle.
\end{abstract}

\section{Introduction}

\emph{Tangles} were introduced by Robertson and Seymour \cite{GMX} as a way to capture `highly cohesive' structures of a graph. 
Formally, a \emph{$k$-tangle} in a graph $G$, for some $k \in \N$, is a consistent orientation of all the separations $\{A,B\}$ of $G$ of order less than $k$, as $(A,B)$ say, such that no three such oriented separations together cover the whole graph by the subgraphs induced on their `small sides' $A$.
The idea is that every highly cohesive substructure of $G$ will lie mostly on one side of every low-order separation, and therefore orients it towards that side, 
thereby defining a tangle.

Since its first introduction, the notion of a tangle and its framework in graphs have been generalized to so-called \emph{abstract separation systems} \cite{ASS}. 
Although these abstract separation systems are significantly more general than the separation system of a graph, the two fundamental theorems about graph tangles -- the `tree-of-tangles' theorem and the `tangle-tree duality' theorem of Robertson and Seymour \cite{GMX} -- are still valid in this setting. We shall also prove most of our results in this more general setting.
	
In what follows we assume familiarity only with graph tangles as described in \cite[Ch.~12.5]{DiestelBook16noEE}.
\bigskip


The \emph{tree-of-tangles theorem} asserts the existence of a \td\ $(T,\cV)$ of $G$, with $\cV = (V_t)_{t \in T}$, say, 
which \emph{distinguishes} all the $k$-tangles in $G$: for every pair of $k$-tangles there exists an edge of $T$ 
which induces a separation that is oriented differently by these two $k$-tangles. For this, recall that $G$ reflects the separation properties of $T$: similar as the deletion of an edge $e = \{t_1,t_2\} \in E(T)$ separates $T$ into two subtrees $T_1 \ni t_1$ and $T_2 \ni t_2$, the sets $U_i := \bigcup_{t \in T_i} V_t$, for $i = 1,2$, form a separation $\{U_1, U_2\}$ of $G$.

Since $(T, \cV)$ distinguishes all the $k$-tangles in $G$, each $k$-tangle \emph{lives} in a different part: the part $G[V_t]$ specified by the property that the given $k$-tangle orients all the separations induced by edges of $T$ incident with~$t$ to that side which contains~$V_t$. 
We call the set of these oriented separations the \emph{star associated with~$V_t$}.
Such a tangle-distinguishing \td\ provides information about the overall structure of the graph and the location of the tangles inside it. A part $V_t \in \cV$ is called \emph{essential} if there is a tangle living in it, and otherwise \emph{inessential}.

In general, the inessential parts might contain a large portion of $G$. The \td\ then tells us nothing about the structure of this portion. 
However, if there are no $k$-tangles in~$G$ at all, there is another theorem which does tell us something about its structure: the \emph{tangle-tree duality theorem}. This asserts the existence of a \td\ in which each part is too small to be home to a tangle, and which thus witnesses that~$G$ has no $k$-tangles (since any $k$-tangle would have to live in some part). 
Formally, such a \td\ has the property that the stars associated with its parts are elements of $\cT_k$: the collection of all those sets of at most three oriented separations which every $k$-tangle has to avoid.

As inessential parts are not home to any tangles, the tangle-tree duality theorem applies locally to these parts. We thus obtain \tds\ of all the inessential parts in~$\cV$ into smaller parts, too small to accommodate a $k$-tangle. Can all these \tds\ be combined into a single \td\ of~$G$ that refines~$(T, \cV)$? 

In general, this will not be possible. However, Erde \cite{JoshRefining} showed that it can be done if those local \tds\ are constructed carefully for this purpose. The following theorem follows directly from \cite[Lemma~3.1]{JoshRefining}:

\begin{THM}{\rm{(Erde 2017)}}\label{thm:JERefining}
    Let $(\tilde{T}, \tilde{\cV})$ be a \td\ of a graph $G$ which distinguishes all the $k$-tangles in $G$, for some $k \geq 3$, and is such that every separation of $G$ induced by an edge of $\tilde{T}$ efficiently distinguishes some pair of $k$-tangles in $G$.  
	Then there exists a \td\ $(T, \cV)$ of $G$ which refines $(\tilde{T}, \tilde{\cV})$ and is such that every star associated with an inessential part of $(T, \cV)$ is a star in $\cT_k$.
\end{THM}

\noindent See \cref{sec:ASS} for definitions.
\medskip

The procedure of refining $(T, \cV)$ by further decomposing its inessential parts can be seen as a way to describe the structure of those parts in a more precise way. 
However, just as the inessential parts of~$(T, \cV)$ may contain a large portion of $G$, another large portion may be hidden inside its essential parts. The \td\ $(T, \cV)$ does not tell us much about this portion. We know from the fact that these parts are essential that there is a tangle living there, but an essential part of $(T, \cV)$ can also contain many vertices and edges that have nothing to do with that tangle.

To get a better insight into the structure of the essential parts of $(T, \cV)$ too, it would be desirable to decompose them by further \tds\ as well. Overall, it would be useful to have a single \td\ of~$G$ whose inessential and essential parts are both as small as possible. As every essential part is home to a tangle, by definition, and every tangle corresponds to some highly cohesive (but maybe fuzzy) substructure of~$G$ which cannot be divided into small parts, `small' can only be achieved relatively to the tangle living in that part. In other words, the essential parts should be as small as possible so that they can still be home to a tangle.

\begin{figure}[h]
  \centering
  \begin{subfigure}[b]{0.35\linewidth}
  \definecolor{dgreen}{rgb}{0,0.7,0}
\definecolor{grey}{rgb}{0.6,0.6,0.6}
\scalebox{0.63}{%
\begin{tikzpicture}
\draw [line width=1.3pt,color=grey] (7,1) circle (3.5cm);
\draw [line width=1.3pt,color=grey] (10.56,3.71) circle (1.5cm);
\draw [line width=1.3pt,color=grey] (3.86,4.19) circle (1.5cm);
\draw [line width=1.3pt,color=grey] (4.82,-2.85) circle (1.5cm);
\draw [line width=1.3pt] (2.61,3.9)-- (2.98,5.12);
\draw [line width=1.3pt] (2.98,5.12)-- (4.23,5.42);
\draw [line width=1.3pt] (4.23,5.42)-- (5.1,4.49);
\draw [line width=1.3pt] (5.1,4.49)-- (4.73,3.26);
\draw [line width=1.3pt] (4.73,3.26)-- (3.49,2.97);
\draw [line width=1.3pt] (3.49,2.97)-- (2.61,3.9);
\draw [line width=1.3pt] (2.61,3.9)-- (5.1,4.49);
\draw [line width=1.3pt] (5.1,4.49)-- (2.98,5.12);
\draw [line width=1.3pt] (2.98,5.12)-- (4.73,3.26);
\draw [line width=1.3pt] (4.73,3.26)-- (4.23,5.42);
\draw [line width=1.3pt] (4.23,5.42)-- (3.49,2.97);
\draw [line width=1.3pt] (3.49,2.97)-- (2.98,5.12);
\draw [line width=1.3pt] (2.61,3.9)-- (4.73,3.26);
\draw [line width=1.3pt] (2.61,3.9)-- (4.23,5.42);
\draw [line width=1.3pt] (3.49,2.97)-- (5.1,4.49);
\draw [line width=1.3pt,color=dgreen] (6.39,2.48)-- (7.64,2.19);
\draw [line width=1.3pt,color=dgreen] (7.64,2.19)-- (8.01,0.96);
\draw [line width=1.3pt,color=dgreen] (8.01,0.96)-- (7.13,0.03);
\draw [line width=1.3pt,color=dgreen] (7.13,0.03)-- (5.88,0.32);
\draw [line width=1.3pt,color=dgreen] (5.88,0.32)-- (5.51,1.55);
\draw [line width=1.3pt,color=dgreen] (5.51,1.55)-- (6.39,2.48);
\draw [line width=1.3pt,color=dgreen] (6.39,2.48)-- (7.13,0.03);
\draw [line width=1.3pt,color=dgreen] (7.13,0.03)-- (7.64,2.19);
\draw [line width=1.3pt,color=dgreen] (7.64,2.19)-- (5.51,1.55);
\draw [line width=1.3pt,color=dgreen] (5.51,1.55)-- (8.01,0.96);
\draw [line width=1.3pt,color=dgreen] (8.01,0.96)-- (5.88,0.32);
\draw [line width=1.3pt,color=dgreen] (5.88,0.32)-- (6.39,2.48);
\draw [line width=1.3pt,color=dgreen] (6.39,2.48)-- (8.01,0.96);
\draw [line width=1.3pt,color=dgreen] (5.88,0.32)-- (7.64,2.19);
\draw [line width=1.3pt,color=dgreen] (5.51,1.55)-- (7.13,0.03);
\draw [line width=1.3pt] (10.25,4.95)-- (9.33,4.06);
\draw [line width=1.3pt] (9.33,4.06)-- (9.64,2.82);
\draw [line width=1.3pt] (9.64,2.82)-- (10.86,2.47);
\draw [line width=1.3pt] (10.86,2.47)-- (11.79,3.35);
\draw [line width=1.3pt] (11.79,3.35)-- (11.48,4.6);
\draw [line width=1.3pt] (11.48,4.6)-- (10.25,4.95);
\draw [line width=1.3pt] (10.25,4.95)-- (10.86,2.47);
\draw [line width=1.3pt] (10.86,2.47)-- (11.48,4.6);
\draw [line width=1.3pt] (11.48,4.6)-- (9.33,4.06);
\draw [line width=1.3pt] (9.33,4.06)-- (11.79,3.35);
\draw [line width=1.3pt] (11.79,3.35)-- (10.25,4.95);
\draw [line width=1.3pt] (10.25,4.95)-- (9.64,2.82);
\draw [line width=1.3pt] (9.64,2.82)-- (11.48,4.6);
\draw [line width=1.3pt] (9.33,4.06)-- (10.86,2.47);
\draw [line width=1.3pt] (9.64,2.82)-- (11.79,3.35);
\draw [line width=1.3pt] (4.25,-1.7)-- (5.53,-1.78);
\draw [line width=1.3pt] (5.53,-1.78)-- (6.1,-2.93);
\draw [line width=1.3pt] (6.1,-2.93)-- (5.4,-3.99);
\draw [line width=1.3pt] (5.4,-3.99)-- (4.12,-3.92);
\draw [line width=1.3pt] (4.12,-3.92)-- (3.55,-2.77);
\draw [line width=1.3pt] (4.25,-1.7)-- (3.54715605064575,-2.77);
\draw [line width=1.3pt] (3.55,-2.77)-- (6.1,-2.93);
\draw [line width=1.3pt] (6.1,-2.93)-- (4.25,-1.7);
\draw [line width=1.3pt] (4.25,-1.7)-- (5.4,-3.99);
\draw [line width=1.3pt] (5.4,-3.99)-- (3.55,-2.77);
\draw [line width=1.3pt] (3.55,-2.77)-- (5.53,-1.78);
\draw [line width=1.3pt] (5.53,-1.78)-- (4.12,-3.92);
\draw [line width=1.3pt] (4.12,-3.92)-- (6.1,-2.93);
\draw [line width=1.3pt] (4.25,-1.7)-- (4.12,-3.92);
\draw [line width=1.3pt] (5.53,-1.78)-- (5.4,-3.99);
\draw [line width=1.3pt] (7.13,0.03)-- (7.9,-0.84);
\draw [line width=1.3pt] (7.9,-0.84)-- (9.35,-0.67);
\draw [line width=1.3pt] (8.01,0.96)-- (9.15,0.65);
\draw [line width=1.3pt] (7.64,2.19)-- (7.57,3.14);
\draw [line width=1.3pt] (7.57,3.14)-- (8.62,3.02);
\draw [line width=1.3pt] (8.62,3.02)-- (7.64,2.19);
\draw [line width=1.3pt] (7.57,3.14)-- (7,4);
\draw [line width=1.3pt] (8.62,3.02)-- (9.64,2.82);
\draw [line width=1.3pt] (6.34,-1.06)-- (5.53,-1.78);
\draw [line width=1.3pt] (5.51,1.55)-- (4.75,2.34);
\draw [line width=1.3pt] (4.75,2.34)-- (4.73,3.26);
\draw [line width=1.3pt] (4.75,2.34)-- (6.39,2.48);

\draw [rotate around={-30.94:(5.41,-1.85)},line width=2pt,color=blue] (5.41,-1.85) ellipse (0.94cm and 0.3cm);
\draw [rotate around={45.33:(4.67,3.34)},line width=2pt,color=blue] (4.67,3.34) ellipse (0.88cm and 0.25cm);
\draw [rotate around={-54.27:(9.62,3.01)},line width=2pt,color=blue] (9.62,3.01) ellipse (0.97cm and 0.28cm);

\draw [->,line width=1.3pt,color=blue] (9.51,2.68) -- (8.96,2.2);
\draw [->,line width=1.3pt,color=blue] (5.06,3.35) -- (5.64,2.85);
\draw [->,line width=1.3pt,color=blue] (5.36,-1.48) -- (5.73,-0.91);

\draw [line width=1.3pt] (6.34,-1.06)-- (5.88,0.32);
\draw [line width=1.3pt] (5.88,0.32)-- (4.44,0.56);
\draw [color=dgreen](8.2,1.9) node[anchor=north west] {\huge $\tau$};
\draw [color=blue](5.8,4.3) node[anchor=north west] {\huge $\sigma$};
\begin{scriptsize}
\draw [fill=black] (4.73,3.26) circle (2.7pt);
\draw [fill=black] (9.64,2.82) circle (2.7pt);
\draw [fill=black] (5.53,-1.78) circle (2.7pt);
\draw [fill=black] (5.1,4.49) circle (2.7pt);
\draw [fill=black] (4.23,5.42) circle (2.7pt);
\draw [fill=black] (2.98,5.12) circle (2.7pt);
\draw [fill=black] (2.61,3.9) circle (2.7pt);
\draw [fill=black] (3.49,2.97) circle (2.7pt);
\draw [fill=black] (4.25,-1.7) circle (2.7pt);
\draw [fill=black] (3.55,-2.77) circle (2.7pt);
\draw [fill=black] (4.12,-3.92) circle (2.7pt);
\draw [fill=black] (5.4,-3.99) circle (2.7pt);
\draw [fill=black] (6.1,-2.93) circle (2.7pt);
\draw [fill=black] (10.86,2.47) circle (2.7pt);
\draw [fill=black] (11.79,3.35) circle (2.7pt);
\draw [fill=black] (11.48,4.6) circle (2.7pt);
\draw [fill=black] (10.25,4.95) circle (2.7pt);
\draw [fill=black] (9.33,4.06) circle (2.7pt);
\draw [dgreen, fill=dgreen] (7.64,2.19) circle (2.7pt);
\draw [dgreen, fill=dgreen] (6.39,2.48) circle (2.7pt);
\draw [dgreen, fill=dgreen] (5.51,1.55) circle (2.7pt);
\draw [dgreen, fill=dgreen] (5.88,0.32) circle (2.7pt);
\draw [dgreen, fill=dgreen] (7.13,0.03) circle (2.7pt);
\draw [dgreen, fill=dgreen] (8.01,0.96) circle (2.7pt);
\draw [fill=black] (7.9,-0.84) circle (2.7pt);
\draw [fill=black] (9.35,-0.67) circle (2.7pt);
\draw [fill=black] (9.15,0.65) circle (2.7pt);
\draw [fill=black] (7.57,3.14) circle (2.7pt);
\draw [fill=black] (8.62,3.02) circle (2.7pt);
\draw [fill=black] (7,4) circle (2.7pt);
\draw [fill=black] (6.34,-1.06) circle (2.7pt);
\draw [fill=black] (4.75,2.34) circle (2.7pt);
\draw [fill=black] (4.44,0.6) circle (2.7pt);
\end{scriptsize}
\end{tikzpicture}
}%

\subcaption{}
    \label{fig:ToTWithSmallParts1}
  \end{subfigure}
\hspace{14mm}
  \begin{subfigure}[b]{0.35\linewidth}
    \definecolor{dgreen}{rgb}{0,0.7,0}
\definecolor{grey}{rgb}{0.6,0.6,0.6}
\scalebox{0.63}{%
\begin{tikzpicture}
\draw [line width=1.3pt,color=grey] (26.86,3.73) circle (1.5cm);
\draw [line width=1.3pt,color=grey] (20.16,4.22) circle (1.5cm);
\draw [line width=1.3pt,color=grey] (21.13,-2.82) circle (1.5cm);
\draw [line width=1.3pt] (18.91,3.92)-- (19.28,5.15);
\draw [line width=1.3pt] (19.28,5.15)-- (20.53,5.44);
\draw [line width=1.3pt] (20.53,5.44)-- (21.4,4.51);
\draw [line width=1.3pt] (21.4,4.517)-- (21.04,3.28);
\draw [line width=1.3pt] (21.04,3.28)-- (19.79,2.99);
\draw [line width=1.3pt] (19.79,2.99)-- (18.91,3.92);
\draw [line width=1.3pt] (18.91,3.92)-- (21.4,4.51);
\draw [line width=1.3pt] (21.4,4.51)-- (19.28,5.15);
\draw [line width=1.3pt] (19.28,5.15)-- (21.04,3.28);
\draw [line width=1.3pt] (21.04,3.28)-- (20.53,5.44);
\draw [line width=1.3pt] (20.53,5.44)-- (19.79,2.99);
\draw [line width=1.3pt] (19.79,2.99)-- (19.28,5.15);
\draw [line width=1.3pt] (18.91,3.92)-- (21.04,3.28);
\draw [line width=1.3pt] (18.91,3.92)-- (20.53,5.44);
\draw [line width=1.3pt] (19.79,2.99)-- (21.4,4.51);
\draw [line width=1.3pt,color=dgreen] (22.69,2.5)-- (23.94,2.21);
\draw [line width=1.3pt,color=dgreen] (23.94,2.21)-- (24.31,0.99);
\draw [line width=1.3pt,color=dgreen] (24.31,0.99)-- (23.43,0.05);
\draw [line width=1.3pt,color=dgreen] (23.43,0.05)-- (22.18,0.35);
\draw [line width=1.3pt,color=dgreen] (22.18,0.35)-- (21.81,1.57);
\draw [line width=1.3pt,color=dgreen] (21.81,1.57)-- (22.69,2.5);
\draw [line width=1.3pt,color=dgreen] (22.69,2.5)-- (23.43,0.05);
\draw [line width=1.3pt,color=dgreen] (23.43,0.05)-- (23.94,2.21);
\draw [line width=1.3pt,color=dgreen] (23.94,2.21)-- (21.81,1.57);
\draw [line width=1.3pt,color=dgreen] (21.81,1.57)-- (24.31,0.99);
\draw [line width=1.3pt,color=dgreen] (24.31,0.99)-- (22.18,0.35);
\draw [line width=1.3pt,color=dgreen] (22.18,0.35)-- (22.69,2.5);
\draw [line width=1.3pt,color=dgreen] (22.69,2.5)-- (24.31,0.99);
\draw [line width=1.3pt,color=dgreen] (22.18,0.35)-- (23.94,2.21);
\draw [line width=1.3pt,color=dgreen] (21.81,1.57)-- (23.43,0.05);
\draw [line width=1.3pt] (26.55,4.98)-- (25.63,4.09);
\draw [line width=1.3pt] (25.63,4.09)-- (25.94,2.85);
\draw [line width=1.3pt] (25.94,2.85)-- (27.17,2.49);
\draw [line width=1.3pt] (27.17,2.49)-- (28.09,3.38);
\draw [line width=1.3pt] (28.09,3.38)-- (27.78,4.62);
\draw [line width=1.3pt] (27.78,4.62)-- (26.55,4.98);
\draw [line width=1.3pt] (26.55,4.98)-- (27.17,2.49);
\draw [line width=1.3pt] (27.166,2.49)-- (27.78,4.62);
\draw [line width=1.3pt] (27.78,4.62)-- (25.63,4.09);
\draw [line width=1.3pt] (25.63,4.09)-- (28.09,3.38);
\draw [line width=1.3pt] (28.09,3.38)-- (26.55,4.98);
\draw [line width=1.3pt] (26.55,4.98)-- (25.94,2.85);
\draw [line width=1.3pt] (25.94,2.85)-- (27.78,4.62);
\draw [line width=1.3pt] (25.63,4.09)-- (27.17,2.49);
\draw [line width=1.3pt] (25.94,2.85)-- (28.09,3.38);
\draw [line width=1.3pt] (20.55,-1.68)-- (21.83,-1.76);
\draw [line width=1.3pt] (21.83,-1.76)-- (22.4,-2.9);
\draw [line width=1.3pt] (22.4,-2.9)-- (21.7,-3.97);
\draw [line width=1.3pt] (21.7,-3.97)-- (20.42,-3.89);
\draw [line width=1.3pt] (20.42,-3.89)-- (19.85,-2.75);
\draw [line width=1.3pt] (20.55,-1.68)-- (19.85,-2.75);
\draw [line width=1.3pt] (19.85,-2.75)-- (22.4,-2.9);
\draw [line width=1.3pt] (22.4,-2.9)-- (20.55,-1.68);
\draw [line width=1.3pt] (20.55,-1.68)-- (21.7,-3.97);
\draw [line width=1.3pt] (21.7,-3.97)-- (19.85,-2.75);
\draw [line width=1.3pt] (19.85,-2.75)-- (21.83,-1.76);
\draw [line width=1.3pt] (21.83,-1.76)-- (20.42,-3.89);
\draw [line width=1.3pt] (20.42,-3.89)-- (22.4,-2.9);
\draw [line width=1.3pt] (20.55,-1.68)-- (20.42,-3.89);
\draw [line width=1.3pt] (21.83,-1.76)-- (21.7,-3.97);
\draw [line width=1.3pt] (23.43,0.05)-- (24.2,-0.82);
\draw [line width=1.3pt] (24.2,-0.82)-- (25.66,-0.65);
\draw [line width=1.3pt] (24.31,0.99)-- (25.45,0.68);
\draw [line width=1.3pt] (23.94,2.21)-- (23.87,3.16);
\draw [line width=1.3pt] (23.87,3.16)-- (24.92,3.05);
\draw [line width=1.3pt] (24.92,3.05)-- (23.94,2.21);
\draw [line width=1.3pt] (23.87,3.16)-- (23.3,4.03);
\draw [line width=1.3pt] (24.92,3.05)-- (25.94,2.85);
\draw [line width=1.3pt] (22.64,-1.03)-- (21.83,-1.76);
\draw [line width=1.3pt] (21.81,1.57)-- (21.05,2.37);
\draw [line width=1.3pt] (21.05,2.37)-- (21.04,3.28);
\draw [line width=1.3pt] (21.05,2.37)-- (22.69,2.5);
\draw [line width=1.3pt] (22.64,-1.03)-- (22.18,0.35);
\draw [line width=1.3pt] (22.18,0.35)-- (20.74,0.62);
\draw [rotate around={-89.35:(21.04,2.83)},line width=1.3pt,color=grey] (21.04,2.83) ellipse (0.78cm and 0.2cm);
\draw [rotate around={-55.03:(23.56,3.62)},line width=1.3pt,color=grey] (23.56,3.62) ellipse (0.8cm and 0.23cm);
\draw [rotate around={9.64:(24.33,2.74)},line width=1.3pt,color=grey] (24.33,2.74) ellipse (0.94cm and 0.72cm);
\draw [rotate around={25.59:(21.74,2.16)},line width=1.3pt,color=grey] (21.74,2.16) ellipse (1.21cm and 0.77cm);
\draw [rotate around={-11.38:(21.47,0.49)},line width=1.3pt,color=grey] (21.47,0.49) ellipse (0.99cm and 0.24cm);
\draw [rotate around={-15.4:(24.86,0.84)},line width=1.3pt,color=grey] (24.86,0.84) ellipse (0.85cm and 0.3cm);
\draw [rotate around={-71.87:(22.41,-0.3)},line width=1.3pt,color=grey] (22.41,-0.3) ellipse (1cm and 0.25cm);
\draw [rotate around={-49.9:(23.81,-0.35)},line width=1.3pt,color=grey] (23.81,-0.35) ellipse (0.84cm and 0.28cm);
\draw [rotate around={7.62:(24.92,-0.7)},line width=1.3pt,color=grey] (24.92,-0.7) ellipse (0.97cm and 0.32cm);
\draw [rotate around={41.32:(22.21,-1.42)},line width=1.3pt,color=grey] (22.21,-1.42) ellipse (0.75cm and 0.26cm);
\draw [rotate around={-11.94:(25.43,2.93)},line width=1.3pt,color=grey] (25.43,2.93) ellipse (0.77cm and 0.27cm);
\draw [line width=1.3pt,color=grey] (23.06,1.28) circle (1.53cm);

\draw [rotate around={48.57:(22.22,2.03)},line width=2pt,color=blue] (22.22,2.03) ellipse (0.86cm and 0.31cm);
\draw [rotate around={-37.57:(23.93,2.28)},line width=2pt,color=blue] (23.93,2.28) ellipse (0.56cm and 0.25cm);
\draw [rotate around={-33.34:(22.21,0.32)},line width=2pt,color=blue] (22.21,0.32) ellipse (0.5cm and 0.22cm);
\draw [rotate around={27.9:(23.48,0.06)},line width=2pt,color=blue] (23.48,0.06) ellipse (0.39cm and 0.22cm);


\draw [rotate around={77.6:(24.3,1)},line width=2pt,color=blue] (24.3,1) ellipse (0.38cm and 0.23cm);

\draw [color=dgreen] (24.7,1.9) node[anchor=north west] {\huge $\tau$};
\draw [color=blue](22.7,3.5) node[anchor=north west] {\huge $\sigma$};
\begin{scriptsize}
\draw [fill=black] (18.91,3.92) circle (2.7pt);
\draw [fill=black] (19.28,5.15) circle (2.7pt);
\draw [fill=black] (20.53,5.44) circle (2.7pt);
\draw [fill=black] (21.4,4.51) circle (2.7pt);
\draw [fill=black] (21.04,3.28) circle (2.7pt);
\draw [fill=black] (19.79,2.99) circle (2.7pt);
\draw [dgreen, fill=dgreen] (22.69,2.5) circle (2.7pt);
\draw [dgreen, fill=dgreen] (23.94,2.21) circle (2.7pt);
\draw [dgreen, fill=dgreen] (23.94,2.21) circle (2.7pt);
\draw [dgreen, fill=dgreen] (23.43,0.05) circle (2.7pt);
\draw [dgreen, fill=dgreen] (22.18,0.35) circle (2.7pt);
\draw [dgreen, fill=dgreen] (21.81,1.57) circle (2.7pt);
\draw [dgreen, fill=dgreen] (24.31,0.99) circle (2.7pt);
\draw [fill=black] (20.74,0.62) circle (2.7pt);
\draw [fill=black] (26.55,4.98) circle (2.7pt);
\draw [fill=black] (25.63,4.09) circle (2.7pt);
\draw [fill=black] (25.94,2.85) circle (2.7pt);
\draw [fill=black] (27.17,2.49) circle (2.7pt);
\draw [fill=black] (28.09,3.38) circle (2.7pt);
\draw [fill=black] (27.78,4.62) circle (2.7pt);
\draw [fill=black] (20.55,-1.68) circle (2.7pt);
\draw [fill=black] (21.83,-1.76) circle (2.7pt);
\draw [fill=black] (22.4,-2.9) circle (2.7pt);
\draw [fill=black] (21.7,-4) circle (2.7pt);
\draw [fill=black] (20.42,-3.89) circle (2.7pt);
\draw [fill=black] (19.85,-2.75) circle (2.7pt);
\draw [fill=black] (24.2,-0.82) circle (2.7pt);
\draw [fill=black] (25.66,-0.65) circle (2.7pt);
\draw [fill=black] (25.45,0.68) circle (2.7pt);
\draw [fill=black] (23.87,3.16) circle (2.7pt);
\draw [fill=black] (24.92,3.05) circle (2.7pt);
\draw [fill=black] (23.3,4.03) circle (2.7pt);
\draw [fill=black] (22.64,-1.03) circle (2.7pt);
\draw [fill=black] (21.05,2.37) circle (2.7pt);
\draw [fill=black] (21.04,3.28) circle (2.7pt);
\end{scriptsize}
\end{tikzpicture}
}%

\subcaption{}
    \label{fig:ToTWithSmallParts2}
  \end{subfigure}
  \vspace{-5mm}
  \centering
  \caption{A tree-decomposition of $G$ that distinguishes all its tangles,
  and a refinement of its central essential part.}
  \label{fig:ToTWithSmallParts}
\end{figure}

For example, consider the graph $G$ in \cref{fig:ToTWithSmallParts}. It contains a central $K_6$, which induces a $3$-tangle~$\tau$ of~$G$: the set of all separations of $G$ of order at most $2$ oriented towards the central $K_6$. The star of separations $\sigma \subseteq \tau$ indicated on the left side of in the figure, which distinguishes~$\tau$ from the other $3$-tangles of $G$ (those induced by the other three~$K_6$ subgraphs), lies far away from the central $K_6$, the `essence' of $\tau$. This results in an essential part which contains a lot more vertices than those contained in the central~$K_6$. But although $\tau$ lives in that part, those other vertices do not really `belong to $\tau$'. 

In the \td\ on the right side of \cref{fig:ToTWithSmallParts} the star $\sigma$ has been moved closer to the core of $\tau$, the $K_6$ that induces it. In this decomposition of~$G$, every essential part contains only the vertices of its $K_6$, while every inessential part is small in that it contains only few vertices.
Clearly, the \td\ on the right side of \cref{fig:ToTWithSmallParts} captures the structure of $G$ in a more precise way than the \td\ on the left side does, and thus provides more information about the graph.
\medskip

Erde \cite{JoshRefining} already introduced a method to refine the essential parts of a given \td\ so that the newly arising essential parts do not contain vertices that are `inessentially separated' from the tangle living in that part. This result comes somewhat close to our goal of decreasing the size of the essential parts by refining them. 
However, it is not strong enough to decrease the size of the essential parts as much as possible, so that even in the example described above, where the tangle $\tau$ is induced by a clique, there are several vertices not in the clique which are not `inessentially separated' from $\tau$. Thus, if we apply the method from~\cite{JoshRefining} to the essential part in \cref{fig:ToTWithSmallParts1} that accommodates $\tau$, then the central essential part of the refined \td\ will still contain more vertices than just those in the clique.

In this paper we show that the essential parts of tangle-distinguishing \tds\ can be refined in the best possible way: so that the newly arising essential parts contain as few vertices as possible while being home to a tangle (see \cref{sec:ASS} for definitions): 

\begin{THM}\label{thm:GraphTanglesToTTDVersion}
    Let $(\tilde{T}, \tilde{\cV})$ be a \td\ of a graph $G$ which distinguishes all the $k$-tangles in~$G$, for some $k \in \N$ so that every separation induced by an edge of $\tilde{T}$ efficiently distinguishes some pair of $k$-tangles in $G$.  
	Then there exists a \td\ $(T, \cV)$ of~$G$ which refines $(\tilde{T}, \tilde{\cV})$ and is such that
    \begin{enumerate}[label=\rm{(\roman*)}]
        \item\label{itm:GraphTanglesToTTDiness} every star associated with an inessential part of $(T, \cV)$ is a star in $\cT_k$;
        \item\label{itm:GraphTanglesToTTDess} every essential part $V_t$ of $(T, \cV)$ is of smallest size among all the exclusive parts of \tds\ of $G$ that are home to the $k$-tangle living in $V_t$.
    \end{enumerate}
\end{THM}

Additionally, the \td\ $(T, \cV)$ from \cref{thm:GraphTanglesToTTDVersion} has the following three properties. First, each of its inessential parts has size $\leq 3k-3$. Second, the unique $k$-tangle $\tau$ that lives in an essential part~$V_t$ of size $> 3k-3$ is induced by $V_t$, in the sense that a separation $\{A,B\}$ of~$G$ of order less than $k$ is oriented by~$\tau$ as $(A,B)$ if and only if $\abs{A \cap V_t} < \abs{B \cap V_t}$. Third, the restriction of $\tau$ to~$G[V_t]$ is a $k$-tangle of $G[V_t]$; thus, $G[V_t]$ alone already witnesses that $\tau$ is indeed a $k$-tangle in $G$.

If a tangle $\tau$ in $G$ is induced by a separable $k$-block in $G$ (see \cref{sec:RefiningToTsInGraphs}), then the essential part of our \td\ that is home to $\tau$ will be equal to that $k$-block. Thus, the \td\ $(T, \cV)$ of~$G$ isolates all the separable $k$-blocks that induce $k$-tangles, in that they appear as parts of the \td. See \cref{thm:TreeDecompWithNiceProperties} for details. 
\medskip

We also prove a variant of \cref{thm:GraphTanglesToTTDVersion} for abstract separation systems that are not induced by a graph. For this, recall that a \td\ of a graph induces a `nested' set of separations of that graph and vice versa. 
Moreover, one can generalize the notion of a $k$-tangle to that of an $\cF$-tangle by allowing other sets $\cF$, instead of $\cT_k$, whose elements an $\cF$-tangle needs to avoid. We will in fact prove the following more general version of \cref{thm:GraphTanglesToTTDVersion} (see \cref{sec:ASS} for definitions):

\renewcommand{\thetheorem}{1.\arabic{theorem}'}
\setcounter{theorem}{1}

\begin{THM}\label{thm:GraphTanglesToTNestedSetVersion}
    Let $G$ be a graph, $k \in \N$, and let $\cF$ be a friendly set of stars in $\vS_k(G)$. 
    Further, let $\tilde{N} \subseteq S_k(G)$ be a nested set of separations that distinguishes all the $\cF$-tangles of~$S_k(G)$ so that every separation in~$\tilde{N}$ efficiently distinguishes a pair of $\cF$-tangles of $S_k(G)$. Then there exists a nested set $N \subseteq S_k(G)$ with $\tilde{N} \subseteq N$ such that
    \begin{enumerate}[label=\rm{(\roman*)}]
        \item\label{itm:GraphTanglesToTNSiness} every inessential node of $N$ is a star in $\cF$;
        \item\label{itm:GraphTanglesToTNSess} the interior of every essential node of $N$ is of smallest size among all the exclusive stars contained in the $\cF$-tangle living at that node.
    \end{enumerate}
\end{THM}

\renewcommand{\thetheorem}{\arabic{section}.\arabic{theorem}}
\setcounter{theorem}{2}

\cref{thm:GraphTanglesToTNestedSetVersion} extends verbatim to abstract separation systems except for its condition~\ref{itm:GraphTanglesToTNSess}, whose reference to `size' involves vertices. Vertices need not exist in the setting of abstract separation systems, not even in the matroid setting. The following variant of \cref{thm:GraphTanglesToTNestedSetVersion} adapts \ref{itm:GraphTanglesToTNSess} to this more general setting while leaving \ref{itm:GraphTanglesToTNSiness} unchanged. Note that the maximality in \ref{itm:GraphTanglesToTTDess} now refers to the nodes of $N$ itself, which are sets of oriented separations contained in the relevant tangle. This maximality translates to the minimality of the interior of $N$ referred to in \ref{itm:GraphTanglesToTNSess} of \cref{thm:GraphTanglesToTNestedSetVersion}, which in the abstract setting is no longer defined.

\begin{THM}\label{thm:abstracttanglesToT}
	Let $\vS$ be a submodular separation system inside some distributive universe and let 
	\[
	\cF \supseteq \big\{\{\vr, \sv\} \subseteq \vS : \vr \leq \vs \text{ and } \vr \vee \sv \text{ is co-small }\big\}
	\]
	be a friendly set of stars in $\vS$. Further, let $\tilde{N}$ be a nested set of good separations that distinguishes all the $\cF$-tangles of $S$. Then there exists a nested set $N \subseteq S$ with $\tilde{N} \subseteq N$ such that
	\begin{enumerate}[label=\rm{(\roman*)}]
        \item\label{itm:abstracttanglesToTiness} every inessential node of $N$ is a star in $\cF$;
		\item\label{itm:abstracttanglesToTess} every essential node of $N$ is maximal in the $\cF$-tangle living at it.
	\end{enumerate}
\end{THM}

\noindent 
\cref{thm:abstracttanglesToT} requires that every separation in~$\tilde{N}$ is good rather than that it efficiently distinguishes two tangles as in \cref{thm:GraphTanglesToTTDVersion}.
This is a necessary adjustment as the notion of `efficiency' does not extend to abstract separation systems. In separation systems induced by a graph these two properties agree.

Condition \ref{itm:abstracttanglesToTess} in \cref{thm:abstracttanglesToT} implies that one cannot refine the nested set~$N$ from \cref{thm:abstracttanglesToT} any further, so~$N$ is optimal as claimed.
We remark that by \cite[Theorem~5.8]{SARefiningInessParts} there always exists a nested set~$\tilde{N}$ which satisfies the premise of \cref{thm:abstracttanglesToT}.
\medskip

This paper is organised as follows. In \cref{sec:ASS} we give a short introduction to abstract separation systems. We then recall the definitions and lemmas from \cite{SARefiningInessParts} about refining inessential stars in \cref{sec:RefiningInessStars}. In \cref{sec:RefiningToTsInGraphs} and \ref{sec:RefiningEssentialStars} we prove \cref{thm:GraphTanglesToTTDVersion} and \ref{thm:abstracttanglesToT}, respectively.

\section{Abstract separation systems and profiles}\label{sec:ASS}

In this section we give a short overview of the important objects, definitions and theorems that we are going to use later. For a more detailed introduction to abstract separation systems and their tangles we refer the reader to \cite{AbstractTangles} and \cite{ASS}.
The definitions are taken from \cite{ASS,TreeSets,AbstractTangles,ProfilesNew,TangleTreeAbstract}. We use the basic graph-theoretic notions from \cite{DiestelBook16noEE}. In particular, we may speak of a vertex $v \in G$ (rather than $v \in V(G)$).

\subsection{Separation systems}\label{subsec:SepSys}

A \emph{separation} of a graph $G$ is a set $\{A, B\}$ of subsets of $V(G)$ such that $A \cup B = V(G)$ and there is no edge in $G$ between $A\setminus B$ and $B \setminus A$. 
The ordered pairs $(A, B)$ and $(B,A)$ of a separation $\{A, B\}$ are its \emph{orientations}. 

The set $\vU(G)$ of all oriented separations of $G$ is partially ordered by
$(A,B) \leq (C,D)$ if $A \subseteq C$ and~$D \subseteq B$.
Inverting the orientations of a separation $\{A,B\}$ is an involution $(A,B) \mapsto (A,B)^\ast := (B,A)$ on~$\vU(G)$ which reverses the partial order, i.e.\ $(A,B) \leq (C,D) \iff (B,A) \geq (D,C)$. 

Any two separations $(A,B), (C,D)$ of $G$ have an infimum $(A \cap C, B \cup D)$ and a supremum $(A \cup C, B \cap D)$ in $\vU(G)$. Hence, $\vU(G)$ forms a lattice together with the operations $(A,B) \wedge (C,D) := (A \cap C, B \cup D)$ and $(A,B) \vee (C,D) := (A \cup C, B \cap D)$.

The number $\abs{A \cap B}$ is the \emph{order} of the separation $\{A,B\}$. It is easy two check that the orders of the infimum and supremum of any two separations $(A,B), (C,D)$ of $G$ sum up to
\[
\abs{(A,B) \vee (C,D)} + \abs{(A,B) \wedge (C,D)} = \abs{(A,B)} + \abs{(C,D)}.
\]
The important part of this equality is that the left side is never larger than the right side, which is called \emph{submodularity}.
\bigskip

A \emph{separation system} $\vS = (\vS, \leq, ^*)$ is a partially ordered set $\vS$ with an order-reversing involution $^*$, i.e.\ if $\vs \leq \vr \in \vS$, then $(\vs)^* \geq (\vr)^*$. We denote $(\vs)^*$ by $\sv$ and $(\sv)^*$ by $\vs$. The elements $\vs, \sv \in \vS$ are called \emph{oriented separations}. An \emph{unoriented} separation is a set $s := \{\vs, \sv\}$ for some $\vs \in \vS$, and $\vs$, $\sv$ are the \emph{orientations} of~$s$. Note that there are no default orientations: Once we denoted one orientation by~$\vs$ the other one will be~$\sv$, and vice versa. The set of all sets $\{\vs, \sv\} \subseteq \vS$ is denoted by~$S$. We will use terms defined for unoriented separations also for oriented ones and vice versa if that is possible without causing ambiguities. Moreover, if the context is clear, we will simply refer to both oriented and unoriented separations as `separations'.

Two unoriented separations $r, s \in S$ are \emph{nested} if they have orientations that can be compared using $\leq$, otherwise they \emph{cross}. 
If $r, s \in S$ are nested, then both $r$ and $\vr$ are \emph{nested} with both of $s, \vs$. Analogously, if $r, s \in S$ cross, then both $r, \vr$ \emph{cross} both of $s, \vs$.
A set of separations is \emph{nested} if all its elements are pairwise nested.
Two sets $R, R'$ of separations are \emph{nested} if every element of $R$ is nested with every element of $R'$.

If a separation $s \in S$ satisfies $\vs = \sv$ we call $s$ \emph{degenerate}. 
A separation $\vs \in \vS$ is called \emph{small} and $\sv$ \emph{co-small} if $\vs \leq \sv$. In the case of separations of a graph $G$ these are exactly the separations of the form $(A,V(G))$.
A separation $\vs \in \vS$ is \emph{trivial} in $\vS$ if there exists a separation $\vr \in \vS$ such that $\vs < \vr$ and $\vs < \rv$. Note that trivial separations are small. Conversely, all but the largest small separations (with respect to~$\leq$) are trivial.
A set $\cF \subseteq 2^{\vS}$ of subsets of $\vS$ is \emph{standard (for $\vS$)} if~$\{\rv\} \in \cF$ for every trivial separation $\vr \in \vS$.

A nested set $\sigma \subseteq \vS$ without degenerate separations is called a \emph{star} if for any $\vr, \vs \in \sigma$ it holds that $\vr \leq \sv$. If $G$ is a graph and $\sigma \subseteq \vU(G)$ a star, then $\interior(\sigma) := \bigcap_{(A,B) \in \sigma} B$ is the \emph{interior of $\sigma$}.

A star $\sigma$ is called \emph{proper} if for every distinct $\vr, \vs \in \sigma$ the relation $\vr \leq \sv$ is the only one, i.e.\ $\vr \not\leq \vs$, $\vr \not\geq \vs$ and $\vr \not\geq \sv$. 
A star $\sigma$ is \emph{proper in $\vS$} if $\sigma$ is proper and $\sigma \neq \{\sv\}$ for a trivial separation $\vs \in \vS$.

The partial order on $\vS$ also relates the proper stars in $\vS$: if $\sigma, \tau \subseteq \vS$ are proper stars, then $\sigma \leq \tau$ if and only if for every $\vs \in \sigma$ there exists some $\vr \in \tau$ such that $\vs \leq \vr$. 
Note that this relation is again a partial order \cite{TreeSets}.
A proper star in $\vS$ is \emph{maximal} if it is a maximal element in the set of all proper stars in $\vS$ with the partial order defined above.

We call a separation system $(\vU, \leq, ^*)$ a \emph{universe} of separations, and denote it with $\vU = (\vU, \leq, ^*, \wedge, \vee)$, if it is a lattice, i.e.\ if any two separations $\vr, \vs \in \vU$ have a \emph{supremum} $\vr \vee \vs$ and an \emph{infimum} $\vr \wedge \vs$ in $\vU$.

A universe $\vU$ of separations is \emph{distributive} if it is distributive as a lattice, i.e.\ if 
\[
\vr \vee (\vs \wedge \vt) = (\vr \vee \vs) \wedge (\vr \vee \vt) \text{ and } \vr \wedge (\vs \vee \vt) = (\vr \wedge \vs) \vee (\vr \wedge \vt)
\]
for all $\vr, \vs, \vt \in \vU$.

Given two separations $r, s \in U$ we call the four separations $\{(\vr \vee \vs), (\vr \vee \vs)^*\}$, $\{(\vr \wedge \vs)$, $(\vr \wedge \vs)^*\}$, $\{(\rv \vee \vs), (\rv \vee \vs)^*\}$ and $\{(\rv \wedge \vs), (\rv \wedge \vs)^*\}$ their \emph{corner separations}. A simple but quite useful observation about corner separations is the following:

\begin{LEM}{\cite[Lemma~3.2]{ASS}} \label{lem:Fischlemma}
	Let $r, s \in S$ be two crossing separations. Every separation
	$t$ that is nested with both $r$ and $s$ is also nested with all four corner separations of $r$ and $s$.
\end{LEM}

A function $\vert\! \cdot\! \vert: \vU \rightarrow \R_{\geq 0}$ is called an \emph{order function} of $\vU$ if $\abs{\vs} = \abs{\sv}$ for all $s \in U$. If $\vU$ comes with an order function that is \emph{submodular}, i.e.\
\[
\abs{\vr \vee \vs} + \abs{\vr \wedge \vs} \leq \abs{\vr} + \abs{\vs} \text{ for all } \vr, \vs \in \vU
\]
we call $\vU$ a \emph{submodular} universe. Then $\abs{\vs}$ is the \emph{order} of $\vs$ and $s$. For an integer $k > 0$, the induced sets~$\vS_k := \{\vs \in \vU: \abs{\vs} < k\}$ are separation systems on their own. Though they need not be universes, as in general $\vr \vee \vs$ and $\vr \wedge \vs$ for two separations $\vr, \vs \in \vS_k$ need not both lie in $\vS_k$. Note that we take the infimum and supremum here with respect to $\vU$. However, by the submodularity of $\vert\! \cdot\! \vert$, at least one of them has to be contained in~$\vS_k$.

This property is rather useful and motivates the following generalization of submodularity which can be applied to universes without making use of the external concept of an order function:

A separation system $\vS$ is \emph{inside} a universe $\vU$ if $\vS \subseteq \vU$ and the partial order and the involution on $\vS$ are the ones induced by $\vU$. 
Further, a separation system $\vS$ inside some universe $\vU$ is called \emph{submodular} if
\[
\text{for all } \vr, \vs \in \vS : \vr \vee \vs \in \vS \text{ or } \vr \wedge \vs \in \vS.
\]

\subsection{Profiles and tangles}\label{sec:profiles}

An \emph{orientation} of a set $S$ of unoriented separations is a set $O \subseteq \vS$ which contains every degenerate separation from $\vS$, and exactly one orientation~$\vs$ or~$\sv$ of every non-degenerate separation in $S$. A subset $O \subseteq \vS$ is \emph{consistent} if it does not contain both~$\rv$ and~$\vs$ whenever $\vr < \vs$ for distinct $r, s \in S$. 
If $\cO$ is a set of consistent orientations of $\vS$, then we call a star~$\sigma \subseteq \vS$ \emph{essential (for $\cO$)} if $\sigma \subseteq O$ for some $O \in \cO$. Otherwise~$\sigma$ is called \emph{inessential (for~$\cO$)}.

A star $\sigma \subseteq \vS$ is \emph{exclusive} for some set $\cO$ of consistent orientations of $S$ if it is contained in exactly one orientation in $\cO$. If $O \in \cO$ is that orientation, we say that $\sigma$ is \emph{$O$-exclusive (for $\cO$)}.

A non-degenerate separation $s \in S$ \emph{distinguishes} two orientations $O_1$ and $O_2$ of $S$ if $O_1$ and $O_2$ orient $s$ differently.
A set of separations $N \subseteq S$ \emph{distinguishes} a set $\cO$ of orientations if any two distinct orientations in $\cO$ are distinguished by some separation in $N$.
If $S$ lies inside a submodular universe, a separation distinguishes a pair of orientations $O_1$ and $O_2$ of $S$ \emph{efficiently} if it distinguishes them and $O_1$ and $O_2$ cannot be distinguished by any separation of lower order.

Let $\cF \subseteq 2^{\vS}$ be a set of subsets of $\vS$. An orientation $P$ of $S$ is an \emph{$\cF$-tangle (of $S$)} if $P$ is consistent and does not contain any element of~$\cF$ as a subset.

If $\vS$ is a separation system inside some universe $\vU$, we call a consistent orientation $P$ of $S$ a \emph{profile of~$S$} if it satisfies that
\[
\text{for all $\vr, \vs \in P$ the separation $(\vr \vee \vs)^*$ does not lie in $P$.}
\]

\noindent A profile of $S$ which contains no co-small separation is \emph{regular}. For the special case of separations of a graph $G$, a \emph{$k$-profile} is a profile of $S_k(G)$.

We can describe the profiles of $S$ as $\cF$-tangles as follows. Set $\cP_S := \{\{\vr, \vs, (\vr \vee \vs)^*\} : \vr, \vs \in \vS\} \cap 2^{\vS}$.
Then the set of all $\cP_S$-tangles of $S$ is exactly the set of all profiles of $S$. We say that $\cF$ is \emph{profile-respecting (for~$\vS$)} if every $\cF$-tangle of $S$ is a profile of~$S$.

The profiles which we consider in this paper will typically be $\cF$-tangles for a standard and profile-respecting set of stars which contains $\{\rv\}$ for every small $\vr \in \vS$, and thus these profiles will be regular.

For the special case of separations of a graph $G$, a \emph{$k$-tangle in $G$} is a $\cT_k$-tangle of $S_k(G)$ where
\[
\cT_k := \{\{(A_1, B_1), (A_2, B_2), (A_3, B_3)\} \subseteq \vS_k(G): G[A_1] \cup G[A_2] \cup G[A_3] = G\}.
\]

\subsection{Tree sets, S-trees and \tds}

A \emph{tree set} is a nested separation system $N$ without degenerate or trivial separations. 
A star $\sigma \subseteq \vN$ is called a \emph{node of} $N$ if there is a consistent orientation $O$ of $N$ such that $\sigma$ is the set of maximal elements in $O$. (In \cite{TreeSets} the nodes of $N$ are called `splitting stars'.) If $\vS$ is a separation system, and if $N \subseteq S$ consists only of separations which each distinguish some pair of consistent orientations of~$S$, then $(\vN, \leq, ^*)$ is a tree set where $\leq$ and $^*$ are the ones induced by $\vS$. Since all relevant nested sets in this paper will be of this form, we will use the term `node' also for nested sets.

Let $\vS$ be again a separation system. An \emph{$S$-tree} is a pair $(T,\alpha)$ of a (graph-theoretical) tree $T$ and a map~$\alpha: \vE(T) \rightarrow \vS$ from the oriented edges $\vE(T)$ of $T$ to $\vS$
such that $\alpha(\ev) = \sv$ if $\alpha(\ve) = \vs$. 
If $x \in V(T)$ is a leaf of $T$ and $t \in V(T)$ its unique neighbour, then we call $\alpha(x,t) \in \vS$ a \emph{leaf separation (of $T$)}.

An $S$-tree $(T, \alpha)$ is \emph{over} a set $\cF \subseteq 2^{\vS}$ if, for every node $t \in T$, we have $\{\alpha(t',t) : (t',t) \in \vE(T)\} \in \cF$. If~$\cF$ is a set of stars in $\vS$, then we call $\sigma_t := \{\alpha(t',t) : (t',t) \in \vE(T)\} \subseteq \vS$ for a node $t \in T$ the \emph{star associated with $t$ (in $T$)}.

Every $S$-tree over a set of stars \emph{induces} a tree set $\vN := \text{im}(\alpha)$ via $\alpha$ if no $\alpha(e)$ is trivial or degenerate for some edge $e \in \vE(T)$.
Conversely, if $N$ is a \emph{regular} tree set in $S$, i.e.\ $\vN$ does not contain any small separations, then one can obtain an $S$-tree $(T,\alpha)$ as follows. We take the set of all nodes of $\vN$ as the vertex set of $T$ and $N$ as the edges of $T$ where we let a separation $s \in N$ be incident to the two nodes of $\vN$ that contain $\vs$ and $\sv$, respectively.
\begin{THM}{\cite[Theorem~6.9]{TreeSets}} \label{thm:FromTreeSetsToSTrees}
    Let $\vS$ be a separation system and $N \subseteq S$ a regular tree set. Then there exists an $S$-tree $(T, \alpha)$ with $\text{\emph{im}}(\alpha) = \vN$ such that the stars associated with nodes of $T$ are precisely the nodes of $\vN$.
\end{THM}
\noindent This motivates the name `nodes' for the splitting stars of $N$. It is shown in \cite{TreeSets} that the $S$-tree from \cref{thm:FromTreeSetsToSTrees} is unique up to isomorphisms.
Therefore, we say that $N$ \emph{induces} the $S$-tree $(T, \alpha)$.

We say that a consistent orientation $O$ \emph{lives} at a node $\sigma$ of $N$ (or equivalently $\sigma$ \emph{is home} to $O$) if $\sigma \subseteq O$. Similarly, we say that $O$ lives at a node $t$ of an $S$-tree $(T, \alpha)$ if $\sigma_t \subseteq O$.
It is easy to see that every consistent orientation of $S$ has to live at a (unique) node of any regular tree set $N$.

Given a set $\cO$ of consistent orientations of $S$ and an $S$-tree $(T, \alpha)$, we call a node $t \in T$ \emph{essential (for~$\cO$)} if there is an orientation in $\cO$ which lives at $t$ and otherwise \emph{inessential (for~$\cO$)}.
\medskip

A \emph{\td}~of a graph $G$ is a pair $(T, \cV)$ of a tree $T$ together with a family $\cV = (V_t)_{t \in T}$ of subsets of $V(G)$ such that $\bigcup_{t\in T} G[V_t] = G$, and such that for every vertex~$v \in G$, the graph~$T[\{t\in T : v\in V_t\}]$ is connected. The sets~$V_t \in \cV$ are called the \emph{parts} of $(T, \cV)$. 
The number $\max\{\abs{V_t \cap V_{t'}} : \{t,t'\} \in E(T)\}$ is the \emph{adhesion} of $(T, \cV)$.

Every edge $(t_1, t_2) \in \vE(T)$ \emph{induces} a separation $\vs_{(t_1,t_2)} := (\bigcup_{t \in T_1} V_t, \bigcup_{t \in T_{2}} V_t)$ of $G$ where $T_1 \ni t_1$ and~$T_{2} \ni t_2$ are the two components of $T-(t_1,t_2)$ \cite{DiestelBook16noEE}.
It is easy to check that the set of separations induced by the edges of a \td\ is nested. Conversely, it follows from \cref{thm:FromTreeSetsToSTrees} that every regular nested set $N$ of separations of a graph $G$ is induced by the edges of some \td\ $(T, \cV)$ of~$G$. In fact, if we require the tree~$T$ to be minimal, then there is a unique such \td\ (up to isomorphisms). We say that $N$ \emph{induces} this \td. In particular, if $(T, \cV)$ is induced by~$N$, then the \emph{stars} $\sigma_t := \{\vs_{(u,t)} : (u, t) \in \vE(T)\}$ \emph{associated with the parts $V_t$} of $(T, \cV)$ are precisely the nodes of~$N$.

Similar as to nested sets, we say that a consistent orientation $O$ of $S_k(G)$, for some $k \in \N$, \emph{lives} in a part $V_t \in \cV$ if~$\sigma_t \subseteq O$. Further, given a set $\cO$ of consistent orientations of $S_k(G)$, we call $V_t$ \emph{essential (for~$\cO$)} if there is an orientation in $\cO$ that lives in $V_t$ and otherwise \emph{inessential (for~$\cO$)}. 

Given two \tds\ $(\tilde{T}, \tilde{\cV})$ and $(T, \cV)$ of $G$, we say that $(T, \cV)$ \emph{refines} $(\tilde{T}, \tilde{\cV})$ if the set $N$ of all the separations induced by the edges of $T$ \emph{refines} the set $\tilde{N}$ induced by the edges of $\tilde{T}$ in that $\tilde{N} \subseteq N$.

\subsection{Friendly sets of stars}

Let $\vS$ be a separation system inside some universe, and let~$\vr$ be a separation in $\vS$ which is neither degenerate nor trivial, and set $S_{\geq \vr} := \{x \in S : \vx \geq \vr \text{ or } \xv \geq \vr\}$. We say that \emph{$\vs$ emulates $\vr$ in $\vS$} if $\vs \geq \vr$ 
and for every $\vx \in \vS$ with $\vx \geq \vr$ we have $\vs \vee \vx \in \vS$. 
We can then define a function
\[
f\!\downarrow_{\vs}^{\vr}: \vS_{\geq \vr} \setminus \{\rv\} \rightarrow \vS_{\geq \vr} \setminus \{\rv\} \text{ by } f\!\downarrow_{\vs}^{\vr}(\vx) := \vx \vee \vs \text { and } f\!\downarrow_{\vs}^{\vr}(\xv) := (\vx \vee \vs)^* \text{ for } \vr \leq \vx.
\]
If $(T,\alpha)$ is an $S$-tree and $\vs$ emulates $\vr$, we can define $\alpha' := \shiftingfct{\vs}{\vr} \circ \alpha$. It is then easy to see that $(T, \alpha')$ is again an $S$-tree, and we call $(T, \alpha')$ the \emph{shift of $(T, \alpha)$ onto $\vs$}.

Given a set of stars $\cF \subseteq 2^{\vS}$ we say that \emph{$\vs$ emulates $\vr$ in $\vS$ for $\cF$} if $\vs$ emulates $\vr$ in $\vS$ and for every star~$\sigma \in \cF$ with $\sigma \subseteq \vS_{\geq \vr} \setminus \{\rv\}$ that contains an element $\vt \geq \vr$ it holds that $f\!\downarrow_{\vs}^{\vr}(\sigma) \in \cF$.

A set $\cF$ of stars in $\vS$ is \emph{closed under shifting} if whenever $\vs \in \vS$ emulates some $\vr \in \vS$, then it also emulates $\vr$ in $\vS$ for $\cF$.
Further, a set $\cF$ of stars in $\vS$ is \emph{friendly} if $\cF$ is standard, profile-respecting, closed under shifting and contains $\{\rv\}$ for every small separation $\vr \in \vS$.

The condition that a set of stars is friendly might seem to be rather strong at first but in practice it can typically be satisfied.
Diestel, Eberenz and Erde \cite{ProfileDuality} showed that any set $\cF \subseteq 2^{\vS}$ can be transformed into a standard set $\hat{\cF}$ of stars which is closed under shifting so that an orientation of $S$ is a regular $\hat{\cF}$-tangle if and only if it is a regular $\cF$-tangle.

\begin{LEM}{\cite[Lemmas 11 \& 14]{ProfileDuality}} \label{lem:MakingASetFFriendly}
    Let $\vS$ be a submodular separation system inside some universe and let $\cF \subseteq 2^{\vS}$ be a standard set. Then there exists a standard set of stars $\hat{\cF} \subseteq 2^{\vS}$ that is closed under shifting such that an orientation $\cO$ of $S$ is a regular $\cF$-tangle if and only if it is a regular $\hat{\cF}$-tangle.
\end{LEM}

Especially, \cref{lem:MakingASetFFriendly} implies that $\cP_S$ can be transformed into a set $\hat{\cP}_S$ of stars which is closed under shifting such that the set of all $\hat{\cP}_S$-tangles of $S$ is precisely the set of all profiles of $S$.
Thus, any set $\cF \subseteq 2^{\vS}$ can be transformed into a friendly set $\bar{\cF} := \hat{\cF} \cup \hat{\cP}_S \cup \{\{\sv\} : \vs \in \vS\text{ is trivial}\}$ of stars in $\vS$ so that an orientation of $S$ is an $\bar{\cF}$-tangle if and only if it is a regular profile and an $\cF$-tangle of $S$.

Moreover, for the special case of a separation system $S_k(G)$ that comes from a graph~$G$, it is easy to check that the set $\cT_k^*$ that consists of all stars in $\cT_k$ is a friendly set of stars. Further, Diestel and Oum have shown that the $\cT_k^*$-tangles of $S_k(G)$ are precisely the $k$-tangles in $G$ if $\abs{G} \geq k$ \cite[Lemma~4.2]{TangleTreeGraphsMatroids}.

\section{Refining inessential stars}\label{sec:RefiningInessStars}

The idea of refining tangle-distinguishing \tds\ has its origin in \cite{JoshRefining}. There, Erde proved that in separation systems of the form $S_k$ inside submodular universes one can, under certain circumstances, refine the inessential nodes of tangle-distinguishing nested sets of separations so that all the inessential nodes of the refinement are `small'. 

Since then, this theorem has been generalized to abstract separation systems \cite{SARefiningInessParts}. 
Much of the theory that was introduced in \cite{SARefiningInessParts}, including the main result which generalizes \cref{thm:JERefining}, will play an important role in the proofs of \cref{thm:GraphTanglesToTTDVersion} and \ref{thm:abstracttanglesToT}. Therefore, in this section we recap those definitions and lemmas.

Let $P$ be a profile of a separation system $S$ inside some universe, and let $\vs \in \vS$ be a separation. We say that $\vs$ is \emph{closely related to~$P$} if $\vs \in P$ and $\vr \wedge \vs \in \vS$ for every $\vr \in P$.
Further, we call a subset of $\vS$ \emph{closely related} to~$P$ if all its elements are closely related to~$P$. 

A separation~$\vr \in P$ is \emph{maximal} in~$P$ if it is a maximal element in~$P$ with respect to the partial order on~$P$ induced by~$\vS$. An example of separations which are closely related to $P$ are those which are maximal in~$P$. In fact, even more is true:
\begin{proposition}{\cite[Lemma~4.4]{SARefiningInessParts}} \label{prop:MaxSepsAreCloselyRelated}
    Let $\vS$ be a submodular separation system, and let $P$ be a profile of $S$. Further, let $Y \subseteq P$ be a nested set such that the inverse $\yv$ of every $\vy \in Y$ is closely related to some profile of~$S$, and let $P_Y \subseteq P$ be the set of all separations in $P$ that are nested with $Y$. Then every maximal separation in $P_Y$ is closely related to $P$.
\end{proposition}

The next observation, which we will use throughout this paper, describes a sufficient condition for a separation to be closely related to some profile:

\begin{proposition}{\cite[Proposition~3.3]{SARefiningInessParts}} \label{prop:ShowingThatASepIsCloselyRelated}
Let $\vS$ be a submodular separation system, and let $\vs \in \vS$ be closely related to some profile $P$ of $S$. Further, let $\vr \in \vS$ with $\vr \leq \vs$, and suppose that $\vr \wedge \vu \in \vS$ for every $\vu \leq \vs$. Then $\vr$ is closely related to $P$.
\end{proposition}

Moreover, infima of separations which are closely related to some profile are, under certain circumstances, again closely related to that profile:

\begin{proposition}{\cite[Proposition~4.3]{SARefiningInessParts}} \label{prop:InfimaOfClRelSetsAreClRel}
    Let $\vS$ be a submodular separation system and $\vs \in \vS$. Further, let $M \subseteq \vS$ be some set of separations such that every~$\vm \in M$ is closely related to some profile $Q_m$ of $S$ which satisfies that $\vs \in Q_m$. Then $\vr := \vs \wedge \bigwedge_{\vm \in M} \vm$ is an element of~$\vS$. Moreover, if $\vs$ is closely related to some profile $P$ of $S$, then $\vr$ is closely related to $P$.
\end{proposition}

The key tool in the proof of \cref{thm:JERefining} is the following lemma which makes it possible to refine inessential stars that consist of separations whose inverses are each closely related to some $\cF$-tangle. 
This lemma is taken from \cite{SARefiningInessParts}; a less general version for submodular universes was first proved in \cite{JoshRefining}.

\begin{LEM}{\cite[Lemma~3.5]{SARefiningInessParts}} \label{lem:reflem}
	Let $\vS$ be a submodular separation system, and let $\cF$ be a friendly set of stars in $\vS$. 
	Further, let $\sigma = \{\vs_1, ..., \vs_k\} \subseteq \vS$ be a star which is inessential for the set of all~$\cF$-tangles of $S$ such that each $\sv_i$ is closely related to some $\cF$-tangle of~$S$. 
	Then there is an $S$-tree over~$\cF \cup \{\{\sv_1\}, ...,\{\sv_k\}\}$ in which each $\vs_i$ appears as a leaf separation. 
\end{LEM}

Erde \cite{JoshRefining} gave an example which shows that it is not possible to refine every inessential star, even if the separation system comes from a graph and $\cF = \cT_k^*$ for some~$k \in \N$. Hence, it is not true that one can refine every nested set which distinguishes all the $\cF$-tangles so that every inessential node of that refinement is a star in~$\cF$. For such a statement to be true, one has to impose further conditions on the nested sets. 

A separation $s \in S$ is \emph{good} (for a set $\cP$ of profiles of $S$) if it distinguishes a pair of profiles $P, P' \in \cP$ so that $\vs \in P$ is closely related to $P$ and $\sv \in P'$ is closely related to $P'$. If we refer to a good set without reference to any set of profiles, then it can be assumed that this set is good for the set of all profiles of~$S$.

We remark that `good' is a generalization of efficiency in that every separation in a submodular universe~$\vU$ which efficiently distinguishes some pair of profiles $P, P'$ in $U$ is good for $\{P, P'\}$ \cite[Proposition~3.4]{SARefiningInessParts}. Conversely, if a separation $s$ of a graph $G$ is good for a pair of regular profiles in $G$, then~$s$ distinguishes these two profiles efficiently \cite[Example~5.1]{SARefiningInessPartsExtended}.

\cref{lem:reflem} implies that the inessential nodes of nested sets of good separations can be refined, which yields the following generalization of \cref{thm:JERefining}:

\begin{THM}{\cite[Theorem~2]{SARefiningInessParts}} \label{thm:RefiningGoodToTs}
    Let $\vS$ be a submodular separation system, $\cF$ be a friendly set of stars in~$\vS$, and let $\cP$ be the set of all $\cF$-tangles of $S$. Further, let $\tilde{N} \subseteq S$ be a nested set that distinguishes all~$\cF$-tangles of $S$, and suppose that every separation in $\tilde{N}$ is good for~$\cP$. Then there exists a nested set~$N \subseteq S$ with $\tilde{N} \subseteq N$ such that every node of $N$ is either a star in $\cF$ or home to an $\cF$-tangle of $S$.
\end{THM}

\section{Refining trees of tangles in graphs}\label{sec:RefiningToTsInGraphs}

Let $G$ be a graph, $k \in \N$, and let $(\tilde{T}, \tilde{\cV})$ be a \td\ of $G$ which distinguishes all its $k$-tangles so that every separation induced by an edge of $\tilde{T}$ efficiently distinguishes some pair of $k$-tangles. Then \cref{thm:JERefining} asserts that $(\tilde{T}, \tilde{\cV})$ extends to a \td\ all whose inessential parts are small. This is done by refining the inessential parts of $(\tilde{T}, \tilde{\cV})$ with further \tds\ that decompose those parts into smaller parts.
In this section we show that we can refine the essential parts with further \tds\ as well so that we then obtain a single \td\ of $G$ whose parts are all relatively small.
The main result of this section is then \cref{thm:GraphTanglesToTTDVersion}, which extends \cref{thm:JERefining} and also strengthens a result of Erde (\cite[Theorem~4.6]{JoshRefining}). 

In order to make the proof of \cref{thm:GraphTanglesToTTDVersion} more straightforward, and to reduce clutter caused by translating back and forth between the \tds\ and the nested sets of separations they induce, we first prove \cref{thm:GraphTanglesToTNestedSetVersion}. Since every \td\ of $G$ induces a nested set of separations of $G$ and vice versa, \cref{thm:GraphTanglesToTTDVersion} then follows immediately.

\begin{customthm}{\ref{thm:GraphTanglesToTNestedSetVersion}}
    \emph{Let $G$ be a graph, $k \in \N$, and let $\cF$ be a friendly set of stars in $\vS_k(G)$.
    Further, let $\tilde{N} \subseteq S_k(G)$ be a nested set that distinguishes all the $\cF$-tangles of $S_k(G)$ so that every separation in~$\tilde{N}$ efficiently distinguishes a pair of $\cF$-tangles of $S_k(G)$. Then there exist a nested set $N \subseteq S_k(G)$ with $\tilde{N} \subseteq N$ such that
    \begin{enumerate}[label=\rm{(\roman*)}]
        \item every inessential node of $N$ is a star in $\cF$;
        \item the interior of every essential node of $N$ is of smallest size among all exclusive stars contained in the $\cF$-tangle living at that node.
    \end{enumerate}}
\end{customthm}

Note that the set $\cT_k^*$ of all stars in $\cT_k$ satisfies the premise of \cref{thm:GraphTanglesToTNestedSetVersion} (see \cref{sec:ASS}), and hence we can apply \cref{thm:GraphTanglesToTNestedSetVersion} to the set of all $k$-tangles in $G$ for some $k \in \mathbb{N}$.
In particular, we can use \cref{thm:GraphTanglesToTNestedSetVersion} to refine the canonical nested set from \cite{CDHH13CanonicalAlg}, which efficiently distinguishes all the $k$-tangles in~$G$:

\begin{COR}\label{cor:GraphTanglesToT}
    Let $G$ be a graph, and $k \in \N$. Then there exist nested sets $\tilde{N} \subseteq N \subseteq S_k(G)$ such that:
    \begin{enumerate}[label=\rm{(\roman*)}]
        \item $\tilde{N}$ is fixed under all automorphisms of $G$ and efficiently distinguishes all $k$-tangles in $G$;
        \item every inessential node of $N$ is a star in $\cT_k$;
        \item the interior of every essential node of $N$ is of smallest size among all exclusive stars contained in the $k$-tangle living at that node. \qed
    \end{enumerate}
\end{COR}

Before we prove \cref{thm:GraphTanglesToTNestedSetVersion}, let us first emphasize the optimality of its condition \ref{itm:GraphTanglesToTNSess}.
Since $\tilde{N}$ distinguishes all $\cF$-tangles of $S_k(G)$, the refinement $N$ does so as well. Hence, each of its essential nodes is necessarily exclusive, and thus their interiors can only be of smallest size among all exclusive stars in the tangle they are home to.
The following example shows that in general there may be non-exclusive stars in a tangle whose interiors contain fewer vertices than the interiors of the exclusive stars in that tangle. Thus, condition~\ref{itm:GraphTanglesToTNSess} cannot be strengthened. 

\begin{example}\label{ex:CounterexampleNonExclusive}
    Let $G$ be the graph depicted in \cref{fig:CounterexNonExclusive1}, which consists of five complete graphs~$K_{20}$ and one $K_8$ that are glued together as depicted in \cref{fig:CounterexNonExclusive1}. Whenever a $K_{20}$ is denoted inside a face of the drawing, we assume that the $K_{20}$ contains the boundary vertices and edges of this face, but no other depicted vertices or edges. Moreover, the $K_8$ contains all vertices in the grey cycle.

    \begin{figure}[h]
        \centering
        \begin{subfigure}[b]{0.35\linewidth}
            \definecolor{lgrey}{rgb}{0.5,0.5,0.5}
\scalebox{0.47}{%
\begin{tikzpicture}
\draw [line width=1.6pt] (3,5)-- (6,5);
\draw [line width=1.6pt] (6,5)-- (9,7);
\draw [line width=1.6pt] (6,5)-- (9,3);
\draw [line width=1.6pt] (9,3)-- (9,7);
\draw [line width=1.6pt] (9,3)-- (9,0);
\draw [line width=1.6pt] (9,7)-- (9,10);
\draw [line width=1.6pt] (9,10)-- (3,5);
\draw [line width=1.6pt] (3,5)-- (9,0);
\draw [line width=1.6pt] (9,7)-- (13,7);
\draw [line width=1.6pt] (9,3)-- (13,3);
\draw [line width=1.6pt] (13,3)-- (9,0);
\draw [line width=1.6pt] (13,3)-- (13,7);
\draw [line width=1.6pt] (13,7)-- (9,10);
\draw [lgrey, line width=1.2pt] plot [smooth cycle, tension=0.3] coordinates {(2.7,5.4) (2.7,4.6) (6,4.5) (9.4,2.3) (9.4,7.67) (6,5.5)};
\draw (10,5.4) node[anchor=north west] {\huge $K := K_{20}$};
\draw (9.8,8.5) node[anchor=north west] {\huge $K_{20}$};
\draw (9.8,1.5) node[anchor=south west] {\huge $K_{20}$};
\draw (6.4,3.4) node[anchor=north west] {\huge $K_{20}$};
\draw (6.4,6.6) node[anchor=south west] {\huge $K_{20}$};
\draw (7.5,5.4) node[anchor=north west] {\huge $K_8$};
\begin{scriptsize}
\draw [fill=black] (3,5) circle (3.2pt);
\draw [fill=black] (6,5) circle (3.2pt);
\draw [fill=black] (9,7) circle (3.2pt);
\draw [fill=black] (9,3) circle (3.2pt);
\draw [fill=black] (9,0) circle (3.2pt);
\draw [fill=black] (9,10) circle (3.2pt);
\draw [fill=black] (13,7) circle (3.2pt);
\draw [fill=black] (13,3) circle (3.2pt);
\draw [fill=black] (7.5,6) circle (3.2pt);
\draw [fill=black] (7.5,4) circle (3.2pt);
\draw [fill=black] (9,5) circle (3.2pt);
\draw [fill=black] (10,7) circle (3.2pt);
\draw [fill=black] (11,7) circle (3.2pt);
\draw [fill=black] (12,7) circle (3.2pt);
\draw [fill=black] (10,3) circle (3.2pt);
\draw [fill=black] (11,3) circle (3.2pt);
\draw [fill=black] (12,3) circle (3.2pt);
\draw [fill=black] (4.5,5) circle (3.2pt);
\draw [fill=black] (9,2) circle (3.2pt);
\draw [fill=black] (9,1) circle (3.2pt);
\draw [fill=black] (9,9) circle (3.2pt);
\draw [fill=black] (9,8) circle (3.2pt);
\end{scriptsize}
\end{tikzpicture}
}%

\subcaption{}
            \label{fig:CounterexNonExclusive1}
        \end{subfigure}
        \hspace{14mm}
        \begin{subfigure}[b]{0.35\linewidth}
            \definecolor{lgrey}{rgb}{0.6,0.6,0.6}
\scalebox{0.47}{%
\begin{tikzpicture}
\draw [line width=1.6pt] (3,5)-- (6,5);
\draw [line width=1.6pt] (6,5)-- (9,7);
\draw [line width=1.6pt] (6,5)-- (9,3);
\draw [line width=1.6pt] (9,3)-- (9,7);
\draw [line width=1.6pt] (9,3)-- (9,0);
\draw [line width=1.6pt] (9,7)-- (9,10);
\draw [line width=1.6pt] (9,10)-- (3,5);
\draw [line width=1.6pt] (3,5)-- (9,0);
\draw [line width=1.6pt] (9,7)-- (13,7);
\draw [line width=1.6pt] (9,3)-- (13,3);
\draw [line width=1.6pt] (13,3)-- (9,0);
\draw [line width=1.6pt] (13,3)-- (13,7);
\draw [line width=1.6pt] (13,7)-- (9,10);
\draw [->,line width=1.6pt,color=red] (11,6.4) -- (11,5.5);
\draw [->,line width=1.6pt,color=blue] (11,3.6) -- (11,4.5);
\draw [red, line width=1.2pt] plot [smooth, tension=0.3] coordinates {(1.8,4.6) (3.7,5.6) (5.9,5.6) (9.4,7.65) (12.3,7.6) (14.2,6.6)};
\draw [red, line width=1.2pt] plot [smooth, tension=0.3] coordinates {(1.8,5.4) (3.7,4.4) (5.9,4.4) (9.5,6.4) (12.6,6.4) (14.2,7.4)};
\draw [blue, line width=1.2pt] plot [smooth, tension=0.3] coordinates {(1.6,5.4) (3.7,4.4) (5.9,4.4) (9.4,2.35) (12.3,2.4) (14.2,3.4)};
\draw [blue, line width=1.2pt] plot [smooth, tension=0.3] coordinates {(1.6,4.6) (3.7,5.6) (5.9,5.6) (9.5,3.6) (12.6,3.6) (14.2,2.6)};
\draw (13.2,8.3) node[red, anchor=north west] {\huge $C_1$};
\draw (13.2,5.7) node[red, anchor=south west] {\huge $D_1$};
\draw (13.2,4.3) node[blue, anchor=north west] {\huge $D_2$};
\draw (13.2,1.7) node[blue, anchor=south west] {\huge $C_2$};
\begin{scriptsize}
\draw [fill=black] (3,5) circle (3.2pt);
\draw [fill=black] (6,5) circle (3.2pt);
\draw [fill=black] (9,7) circle (3.2pt);
\draw [fill=black] (9,3) circle (3.2pt);
\draw [fill=black] (9,0) circle (3.2pt);
\draw [fill=black] (9,10) circle (3.2pt);
\draw [fill=black] (13,7) circle (3.2pt);
\draw [fill=black] (13,3) circle (3.2pt);
\draw [fill=black] (7.5,6) circle (3.2pt);
\draw [fill=black] (7.5,4) circle (3.2pt);
\draw [fill=black] (9,5) circle (3.2pt);
\draw [fill=black] (10,7) circle (3.2pt);
\draw [fill=black] (11,7) circle (3.2pt);
\draw [fill=black] (12,7) circle (3.2pt);
\draw [fill=black] (10,3) circle (3.2pt);
\draw [fill=black] (11,3) circle (3.2pt);
\draw [fill=black] (12,3) circle (3.2pt);
\draw [fill=black] (4.5,5) circle (3.2pt);
\draw [fill=black] (9,2) circle (3.2pt);
\draw [fill=black] (9,1) circle (3.2pt);
\draw [fill=black] (9,9) circle (3.2pt);
\draw [fill=black] (9,8) circle (3.2pt);
\end{scriptsize}
\end{tikzpicture}
}%

\subcaption{}
            \label{fig:CounterexNonExclusive2}
        \end{subfigure}
        \vspace{-5mm}
        \centering
        \caption{\cref{ex:CounterexampleNonExclusive}}
        \label{fig:CounterexNonExclusive}
    \end{figure}

    Let $\tau$ be the $10$-tangle defined by orienting every separation towards that side which contains the clique~$K$. Further, observe that every separation $\{A, B\} \in S_{10}(G)$ with $V(K_8) \subseteq A \cap B$ is of the form $\{A, V(G)\}$. Thus, we obtain a consistent orientation $\tau'$ of $S_{10}(G)$ by letting $(A,V(G)) \in \tau'$ for every set $A \subseteq V(G)$ of at most nine vertices, and $(A,B) \in \tau'$ if $V(K_8) \subseteq B$ for all other separations $\{A,B\} \in S_{10}(G)$. It is straightforward to check that $\tau'$ is a $10$-tangle in $G$. 

    Moreover, it is easy to see that every star $\rho$ in~$\tau$ whose interior is of smallest size among all stars in~$\tau$ consists of $(C_1, D_1), (C_2, D_2)$ (see \cref{fig:CounterexNonExclusive2}) and possibly some separations of the form $(A, V(G))$.
    But these separations are also contained in $\tau'$, and thus $\rho$ is not exclusive for the set of all $10$-tangles in $G$. 
\end{example}

Let us now turn to the proof of \cref{thm:GraphTanglesToTNestedSetVersion}.
By \cref{thm:JERefining}, there exists a nested set $N' \subseteq S_k(G)$ which refines the inessential nodes of $\tilde{N}$ and satisfies \cref{itm:GraphTanglesToTNSiness}.
Therefore, we are left to refine the essential nodes of $N'$ so that this refinement satisfies \cref{itm:GraphTanglesToTNSiness} and \cref{itm:GraphTanglesToTNSess}. 
For this, we show the following lemma about refining essential stars in graphs, which will imply \cref{thm:GraphTanglesToTNestedSetVersion}:

\begin{LEM}\label{lem:ReflemEssStarsGraphs}
    Let $G$ be a graph, $k \in \N$, and let $\cF$ be a friendly set of stars in $\vS_k(G)$. Further, let $\sigma = \{\vs_1, \dots, \vs_n\} \subseteq \vS_k(G)$ be a star which is home to a unique $\cF$-tangle $\tau$ of $S_k(G)$, and suppose that each~$s_i$ efficiently distinguishes some pair of $\cF$-tangles of $S_k(G)$. 
    
    Then there exists a star~$\sigma' \subseteq \tau$ whose interior is of smallest size among all exclusive stars in $\tau$, and an $S_k(G)$-tree over $\cF' := \cF \cup \{\sigma'\} \cup \{\{\sv_1\}, \dots, \{\sv_n\}\}$ in which each~$\vs_i$ appears as a leaf separation.
\end{LEM}

As $\tau$ has to live at a node of any $S_k(G)$-tree over $\cF'$, and since $\sigma'$ is the unique star in $\cF'$ that is contained in $\tau$, every $S_k(G)$-tree over $\cF'$ has to contain a node which is associated with $\sigma'$. Thus, there can only exist an $S_k(G)$-tree over $\cF'$ in which each $\vs_i \in \sigma$ appears as a leaf separation if $\sigma'$ is nested with~$\sigma$. Our first step towards the proof of \cref{lem:ReflemEssStarsGraphs} will thus be to show that under all exclusive stars in $\tau$ whose interiors are of smallest size there is at least one star which is nested with $\sigma$. 

\begin{LEM}\label{lem:StarWithMinIntNested}
    Let $k \in \N$, let $\cP$ be some set of $k$-profiles in a graph $G$, and let $P \in \cP$. Further, let $\sigma \subseteq P$ be a star, and suppose that every separation in $\sigma$ efficiently distinguishes some pair of $k$-profiles in $\cP$. 
    Then there exists a star $\rho \subseteq P$ with $\sigma \leq \rho$ whose interior is of smallest size among all stars in~$P$ that are exclusive for $\cP$.
\end{LEM}

\begin{proof}
    Let $\rho$ be a star in $P$ which is exclusive for $\cP$ and which is nested with as many separations in $\sigma$ as possible such that its interior is of smallest size among all stars in $P$ that are exclusive for $\cP$. 
    We claim that~$\rho$ is in fact nested with~$\sigma$. This then clearly implies the assertion. Indeed, the unique node $\rho'$ of $\{\{A,B\} : (A,B) \in \sigma \cup \rho\}$ which is home to~$P$ then satisfies both $\sigma \leq \rho'$ and $\rho \leq \rho'$, and in particular $\interior(\rho') \subseteq \interior(\rho)$. Hence, $\rho'$ is as desired.
    
    So suppose for a contradiction that there is a separation $(C,D) \in \sigma$ which is not nested with $\rho$. By assumption, $\{C,D\}$ efficiently distinguishes some pair $Q \ni (D,C)$ and $Q'$ of $k$-profiles in $\cP$. 
    Since $Q$ is consistent and because $\rho \not\subseteq Q$ as $\rho$ is $P$-exclusive for $\cP$, there is a unique separation $(A,B) \in \rho$ with $(B,A) \in Q$. Then $(C,D) \wedge (A,B)$ and $(C, D) \wedge (B', A')$, for every $(A', B') \in \rho \setminus \{(A,B)\}$, have order at least $\abs{C \cap D}$ because $\{C,D\}$ distinguishes $Q$ and $Q'$ efficiently. By submodularity, it follows that $(C,D) \vee (B', A')$ has order $\leq \abs{A' \cap B'} < k$, and further 
    \begin{equation}\label{eq:ProofOfStarWithMinIntNested}
    \abs{A \cap B} \geq \abs{(C \cup A) \cap (D \cap B)} = \abs{(A \cap B \cap D) \cup (B \cap C \cap D)} = \abs{A \cap B \cap (D \setminus C)} + \abs{B \cap C \cap D}.
    \end{equation}
    Thus, the star $\rho' := \{(A \cup C, B \cap D)\} \cup \{(A' \cap D, B' \cup C) : (A', B') \in \rho\setminus \{(A,B)\}\}$ is a subset of $\vS_k(G)$. In particular, since $P$ is a profile, we have $\rho' \subseteq P$. Moreover, $\rho'$ is $P$-exclusive for $\cP$ since for every $k$-profile $P' \neq P$ in $\cP$ either $(B,A) \in P'$ or $(D,C) \in P'$ and thus $(B \cap D, A \cup C) \in P'$, or there exists some $(A', B') \in \rho$ such that $(B', A') \in P'$, and we then have $(B' \cup C, A' \cap D) \in P'$ because $P'$ is a profile.
    
    We claim that the interior $X'$ of $\rho'$ has at most the size of the interior $X$ of $\rho$. This then contradicts the choice of $\rho$ since $\rho'$ is nested with $(C,D) \in \sigma$ by construction, and thus, by \cref{lem:Fischlemma}, $\rho'$ is nested with at least one separation in $\sigma$ more than $\rho$. Indeed, we have 
    \[
    X' = (B \cap D) \cap \bigcap_{(A',B') \in \rho\setminus\{(A,B)\}} (B' \cup C) = (X \cap D) \cup (B \cap C \cap D) = (X \cap (D \setminus C))\; \dot\cup\; (B \cap C \cap D).
    \]
    Since $\{C,D\}$ is a separation of $G$, we have $X = (X \cap C)\; \dot\cup\; (X \cap (D \setminus C))$, and thus
    \[
    \abs{X'} = \abs{X} - \abs{X \cap C} + \abs{B \cap C \cap D} \leq \abs{X} - \abs{(A \cap B) \cap C} + \abs{B \cap C \cap D},
    \]
    where we have used that $A \cap B \subseteq X$. Combining this with \eqref{eq:ProofOfStarWithMinIntNested} yields that
    \[
    \abs{X'} \leq \abs{X} - \abs{(A \cap B) \cap C} + (\abs{A \cap B} - \abs{A \cap B \cap (D\setminus C)}) = \abs{X}
    \]
    since $\abs{A \cap B} = \abs{(A \cap B) \cap C} + \abs{A \cap B \cap (D\setminus C)}$, which concludes the proof.
\end{proof}

For the proof of \cref{lem:ReflemEssStarsGraphs} we need one further auxiliary statement, \cref{prop:MinOrderSepsAreCloselyRelated} below, which follows from the following two lemmas that are immediate from \cite[Lemmas~4.2 \&~4.3]{JoshRefining}.

\begin{lemma} \label{lem:ErdeLemma1}
    Let $\vs \leq \vr$ be two separations of some graph $G$.   If $\vt$ is a separation of $G$ of minimal order such that $\vs \leq \vt \leq \vr$, then $|\vt \wedge \vx| \leq |\vx|$ for all $\vx \leq \vr$ .
\end{lemma}

\begin{lemma} \label{lem:ErdeLemma2}
    Let $\vr_1$ and $\vr_2$ be two separations of some graph $G$.  Let $\vs_1$ be any separation of  minimal  order  such  that $\vr_1 \wedge \rv_2 \leq \vs_1 \leq \vr_1$ and set $\vs_2 := \vr_2 \wedge \sv_1$. Then $|\vs_i| \leq |\vr_i|$ for $i \in \{1,2\}$, and for $\vx \leq \vr_i$ we have $|\vs_i \wedge \vx| \leq |\vx|$.
\end{lemma}

\begin{lemma}\label{prop:MinOrderSepsAreCloselyRelated}
    Let $P$ be a $k$-profile in some graph $G$ for some $k \in \N$, and let $\vs \in P$. Further, let $\vs' \in P$ be any separation of minimal order such that $\vs \leq \vs'$. Then $\vs'$ is closely related to $P$, and for every $\vr \in P$ with $\vr \leq \sv$ we have $\vr \wedge \sv' \in \vS_k(G)$. Moreover, if $\vr$ is closely related to $P$, then $\vr \wedge \sv'$ is closely related to~$P$.
\end{lemma}

\begin{proof}
    Let $\vt \in P$ be a maximal separation in $P$ such that $\vs' \leq \vt$. Applying \cref{lem:ErdeLemma1} to $\vs$ and~$\vt$ yields that $\vs' \wedge \vx \in \vS_k(G)$ for every $\vx \leq \vt$, which by \cref{prop:MaxSepsAreCloselyRelated,prop:ShowingThatASepIsCloselyRelated} implies that $\vs'$ is closely related to $P$. Further, applying \cref{lem:ErdeLemma2} to $\vs'$ and $\vr$ yields that $\vr \wedge \sv'$ is an element of $\vS_k(G)$ and that $(\vr \wedge \sv') \wedge \vx \in \vS_k(G)$ for every $\vx \leq \vr$, which by \cref{prop:ShowingThatASepIsCloselyRelated} implies that~$\vr \wedge \sv'$ is closely related to~$P$ if~$\vr$ is closely related to $P$.
\end{proof}

We can now prove \cref{lem:ReflemEssStarsGraphs}:

\begin{proof}[Proof of \cref{lem:ReflemEssStarsGraphs}]
    Let $\cP$ be the set of all $\cF$-tangles of $S_k(G)$, and first assume that there is a star $\sigma' \subseteq \tau$ with $\sigma \leq \sigma'$ whose interior is of smallest size among all stars in $\tau$ that are exclusive for $\cP$, and which has the further property that every separation in~$\sigma'$ is closely related to $\tau$. Then the assertions follows.
    Indeed, set $N' := \{s : \vs \in \sigma \cup \sigma'\}$ and let $\rho$ be an inessential node of $N'$. Then $\rho \neq \sigma'$ since $\sigma'$ is home to $\tau$ and hence essential. In particular,~$\rho$ contains no separation from~$\sigma'$. Thus, by the assumptions on $\sigma$ and~$\sigma'$, every separation in $\rho$ has an inverse that is closely related to an $\cF$-tangle of $S_k(G)$, which implies that~$\rho$ satisfies the premise of \cref{lem:reflem}. So we may apply
    \cref{lem:reflem} to the inessential nodes of $N'$, which yields a refinement~$N$ of $N'$ all whose inessential nodes are stars in $\cF$.
    This set $N$ induces an $S_k(G)$-tree over $\cF \cup \{\sigma'\} \cup \{\{\sv_1\}, \dots, \{\sv_n\}\}$ in which each~$\vs_i$ appears as a leaf separation. By the choice of~$\sigma'$, this clearly implies the assertion.

    Thus, it suffices to find a star $\sigma'$ as above.
    For this, let $\sigma' \subseteq \tau$ be a star such that 
    \begin{enumerate}[label=(\arabic*)]
        \item\label{itm:ProofOfReflemEssStarsGraphs1} $\sigma \leq \sigma'$ and $\abs{\interior(\sigma')} = \min\{\abs{\interior(\rho)} : \rho \subseteq \tau \text{ and } \rho 
        \text{ is exclusive for } \cP\}$, and
        \item\label{itm:ProofOfReflemEssStarsGraphs2} the number of separations in $\sigma'$ that are not closely related to $\tau$ is minimal among all stars in $\tau$ that satisfy \ref{itm:ProofOfReflemEssStarsGraphs1}.
    \end{enumerate}
    Note that, by \cref{lem:StarWithMinIntNested}, there exists a star which satisfies \ref{itm:ProofOfReflemEssStarsGraphs1}. We claim that $\sigma'$ is closely related to $\tau$, which implies that $\sigma'$ is the star for the argument above, and thus concludes the proof.
    For this, suppose for a contradiction that there is a separation $(C', D') \in \sigma'$ which is not closely related to $\tau$, and let $(A, B) \in \tau$ be of minimal order such that $(C', D') \leq (A, B)$. 
    We first show that there is such a separation $\{A,B\}$ which is nested with $\sigma$.

    \begin{claim}\label{claim:ProofOfReflemEssStarsGraphs1}
        \emph{For every $\vr \in \tau$ there exists a separation $\vr' \in \tau$ of order $\leq \abs{r}$ such that $r'$ is nested with~$\sigma$. Moreover, if $\vs \in \tau$ is nested with~$\sigma$ such that $\vs \leq \vr$, then $r'$ can be chosen so that $\vs \leq \vr'$.}
    \end{claim}

    \begin{claimproof}
        Let $\vr' (\geq \vs)$ be a separation of order $\leq \abs{r}$ which is nested with as many separations in $\sigma$ as possible. We claim that $\vr'$ is nested with $\sigma$, which then clearly implies the assertion. For this, suppose for a contradiction that $\vr'$ crosses some $\vt \in \sigma$. By the assumption on~$\sigma$, there exists a pair of $k$-profiles $Q \ni \tv$ and $Q' \ni \vt$ in $G$ such that $t$ distinguishes them efficiently. Since $Q$ is a $k$-profile, it contains an orientation of~$r'$; by symmetry we may assume that $\vr' \in Q$. If $\vr' \vee \tv \in \vS_k(G)$, then $\vr' \vee \tv \in Q$ and $\rv' \wedge \vt \in Q'$ as $Q$ and $Q'$ are $k$-profiles, and thus $\vr \vee \tv$ distinguishes~$Q$ and~$Q'$. Since $t$ efficiently distinguishes~$Q$ and~$Q'$, this implies that $\vr' \vee \tv$ has order $\geq \abs{t}$, and hence, by submodularity, $\vr' \wedge \tv$ has order $\leq \abs{r'}$. Then $\vr' \wedge \tv \in \tau$ because $\tau$ is a profile, which contradicts the choice of $\vr'$ since $\vr' \wedge \tv$ is nested with one separation in $\sigma$ more than $\vr'$ by \cref{lem:Fischlemma} (and still $\vs \leq \vr' \wedge \tv$ as $\vs \leq \vr$ and $s$ is nested with $t$).
    \end{claimproof}

    By \cref{claim:ProofOfReflemEssStarsGraphs1}, we may assume that $\{A,B\}$ is nested with $\sigma$. Then, by applying \cref{prop:MinOrderSepsAreCloselyRelated} to $\vs := (C', D')$ and $\vs' := (A,B)$, we obtain that the star
    \[
    \sigma'' := \{(A,B)\} \cup \{(C \cap B, D \cup A) : (C,D) \in \sigma' \setminus \{(C', D')\}\}
    \]
    is a subset of $\vS_k(G)$, which by the consistency of $\tau$ implies that $\sigma'' \subseteq \tau$. Moreover, we have $\sigma \leq \sigma''$ as $\sigma \leq \sigma'$ and $(A, B)$ is nested with $\sigma$. 
    We claim that~$\sigma''$ witnesses that~$\sigma'$ does not satisfy \cref{itm:ProofOfReflemEssStarsGraphs2}, which then contradicts the choice of~$\sigma'$ and thus concludes the proof. By \cref{prop:MinOrderSepsAreCloselyRelated}, $(A, B)$ is closely related to~$\tau$, and if $(C,D) \in \sigma'$ is closely related to~$\tau$, then $(C \cap B, D \cup A)$ is closely related to $\tau$, too. Thus, there are fewer separations in~$\sigma''$ that are not closely related to~$\tau$ than in~$\sigma'$. Therefore, we are left to show that $\sigma''$ satisfies \ref{itm:ProofOfReflemEssStarsGraphs1}. By the argument above, we have $\sigma \leq \sigma''$, so we only need to show that the interior $X''$ of $\sigma''$ has at most the size of the interior $X'$ of $\sigma'$. By definition, we have
    \[
        X'' = B \cap \bigcap_{(C,D) \in \sigma'} (D \cup A) = B \cap (X' \cup A) = (B \cap X') \cup (B \cap A) = (B \cap X')\; \dot\cup\; ((B \cap A)\setminus X').
    \]
    Since $\{A,B\}$ is a separation of $G$, we have $X' = (X' \cap B)\; \dot\cup\; (X' \cap (A\setminus B))$ and thus
    \[
    \vert X''\vert = \vert X' \vert - \vert (A \setminus B)\cap X'\vert + \vert (A \cap B) \setminus X'\vert.
    \]
    So we are done if $\abs{(A \cap B)\setminus X'} \leq \abs{(A\setminus B) \cap X'}$. Set $\rho := \sigma' \setminus \{(C', D')\}$.
    By the choice of~$(A, B)$, the separation $(A, B) \wedge \bigwedge_{(C,D) \in \rho} (D, C)$ has at least order $\abs{A \cap B}$, and thus
    \[
    \big\vert A \cap B\big\vert \leq \Big\vert A \cap \bigcap_{(C,D) \in \rho} D \cap \Big(B \cup \bigcup_{(C,D) \in \rho} C\Big)\Big\vert = \Big\vert\Big(A \cap B \cap \bigcap_{(C,D) \in \rho} D\Big) \cup \Big(A \cap \bigcap_{(C,D) \in \rho} D \cap \bigcup_{(C,D) \in \rho} C\Big)\Big\vert.
    \]
    Since $\sigma'$ is a star, we have $C \subseteq D'$ for every $(C,D) \in \rho$, and hence $\bigcup_{(C,D) \in \rho} C \subseteq D'$. Moreover, by the choice of $(A,B)$, we have $B \subseteq D'$, and thus
    \[
    \vert A \cap B\big\vert \leq \vert(A \cap B \cap X') \cup (A \cap X')\vert = \vert A \cap B \cap X'\vert + \vert(A\setminus B) \cap X'\vert.
    \]
    Combining this inequality with $\abs{A \cap B} = \abs{A \cap B \cap X'} + \abs{(A \cap B) \setminus X'}$ yields that
    \[
    \vert(A\setminus B) \cap X' \vert \geq \vert A \cap B\vert - \vert A \cap B \cap X'\vert = \vert A \cap B\vert - (\vert A \cap B\vert - \vert(A \cap B) \setminus X'\vert) = \vert(A \cap B) \setminus X'\vert,
    \]
    which concludes the proof.
\end{proof}

With \cref{lem:ReflemEssStarsGraphs} at hand \cref{thm:GraphTanglesToTNestedSetVersion} follows immediately:

\begin{proof}[Proof of \cref{thm:GraphTanglesToTNestedSetVersion}]
    Apply \cref{thm:JERefining} to $\tilde{N}$ to obtain a nested set $N' \subseteq S_k(G)$ with $\tilde{N} \subseteq N'$ which satisfies \cref{itm:GraphTanglesToTNSiness}. Then, apply \cref{lem:ReflemEssStarsGraphs} to the essential nodes of $N'$.
\end{proof}

\begin{proof}[Proof of \cref{thm:GraphTanglesToTTDVersion}]
Apply \cref{thm:GraphTanglesToTNestedSetVersion} to the set of separations induced by $(\tilde{T}, \tilde{\cV})$. 
\end{proof}

Recall that the essential parts of the \td\ from \cref{thm:GraphTanglesToTTDVersion} are as small as possible so that they are still home to their tangle. Thus, this refinement is optimal in that it decreases the size of the essential parts as much as possible. 

Now, one may ask what else could be said about the essential parts of that refinement besides their size being as small as possible. Our main reason for decreasing the size of the parts was to more precisely exhibit where the tangles are located in the graph, so it seems natural to ask whether the vertices in the essential parts in some sense `belong' to the tangle living in that part, and additionally whether all vertices that `belong' to a tangle are contained in the part which is home to that tangle. 

In general, this question cannot be answered because tangles are, by nature, a somewhat fuzzy object, which is, at least in general, not related to a specific set of vertices. So usually it cannot be determined at all whether a certain vertex `belongs' to a tangle or not. Thus, in general, we also cannot evaluate whether the essential parts contain precisely those vertices which `belong' to that tangle.

However, in the following, we present two properties of the refined \td\ from \cref{thm:GraphTanglesToTTDVersion} that we believe validate that its essential parts, at least in many cases, `closely correspond' to the tangle they are home to. 
For this, let $G$ be some graph, and let $(T, \cV)$ be the \td\ which one obtains by applying \cref{thm:GraphTanglesToTTDVersion} to any \td\ of $G$ which efficiently distinguishes all the regular $k$-profiles in $G$. (Such \tds\ exist by \cite[Theorem~3.6]{ProfilesNew}.)

Then each part of $(T, \cV)$ of size $> 3k-3$ is home to a $k$-tangle, which it then `witnesses' and `induces'. More precisely, let us say that a subgraph $H$ of $G$ \emph{witnesses} a $k$-tangle $\tau$ if for every subset $\{(A_i, B_i) : i \in [3]\}$ of~$\tau$ of at most three elements we have $\bigcup_{i \in [3]} G[A_i \cap V(H)] \neq H$. Further, a set~$U$ of vertices of~$G$ \emph{induces} a $k$-profile~$P$ if for every separation $(A,B) \in P$ we have $\abs{A \cap U} < \abs{B \cap U}$.

If a part $V_t \in \cV$ has size $> 3k-3$, then $\sigma_t$ cannot be a star in $\cP_{S_k(G)} \subseteq \cT_k$, and thus by \ref{itm:GraphTanglesToTTDiness}, $V_t$ is home to some regular $k$-profile~$\tau$ in $G$. 
By \ref{itm:GraphTanglesToTTDess} and the following lemma, it then follows that $\abs{A \cap V_t} < k$ for every $(A,B) \in \tau$. 

\begin{LEM}{\cite[Proof of Theorem~12]{Focus}} \label{lem:IntersectionOfSmallSideWithMinStar}
    Let $k \in \N$, let $\cP$ be some set of $k$-profiles in a graph $G$, and let~$P \in \cP$. Further, let $\sigma \subseteq P$ be a star whose interior is of smallest size among all stars in~$P$ that are exclusive for $\cP$. Then $\abs{A \cap \interior(\sigma)} < k$ for all $(A,B) \in P$.
\end{LEM}

\noindent Note that in \cite{Focus} the statement of \cref{lem:IntersectionOfSmallSideWithMinStar} is shown for stars in $P$ whose interior is of smallest size among all stars in $P$. However, it is easy to see that the same proof also yields the statement of \cref{lem:IntersectionOfSmallSideWithMinStar}, where the interior of $\sigma$ is only of smallest size among all \emph{exclusive} stars in $P$.
\medskip

By \cref{lem:IntersectionOfSmallSideWithMinStar} we have $\abs{A \cap V_t} < k \leq (3k-2) - (k-1) \leq \abs{V_t} - \abs{A \cap V_t} = \abs{(B \setminus A) \cap V_t} \leq \abs{B \cap V_t}$, and hence~$V_t$ induces~$\tau$. Moreover, for every subset $\{(A_i, B_i) : i \in [3]\} \subseteq \tau$, we have $\bigcup_{i \in [3]} G[A_i \cap V_t] \neq G[V_t]$ since $\abs{V_t \cap \bigcap_{i \in [3]} (B_i \setminus A_i)} \geq \abs{V_t} - \abs{A_1} - \abs{A_2} - \abs{A_3} \geq (3k-2) - 3(k-1) = 1$. Hence,~$\tau$ is in fact a $k$-tangle in $G$ as witnessed by~$G[V_t]$.

Therefore, every large enough part of $(T, \cV)$ induces a $k$-tangle. Conversely, if a $k$-tangle is induced by some large and highly connected substructure of $G$, does that substructure appear as a part of~$(T, \cV)$? In general, this will not be the case even if a $k$-tangle is induced by a large clique, since even a clique must not be equal to a part of any \td\ of adhesion~$<k$ \cite{CDHH13CanonicalParts}. However, if a $k$-tangle is induced by a `$k$-block', then the $k$-block will be equal to a part of~$(T, \cV)$ if there exists any \td\ of adhesion~$<k$ at all that contains the $k$-block as a part.

For some $k \in \N$, a \emph{$k$-block} in a graph $G$ is a maximal set $b$ of at least $k$ vertices such that no two vertices $v, w \in b$ can be separated in $G$ by fewer than $k$ vertices.
It is straight forward to check that every $k$-block \emph{induces} a regular $k$-profile by orienting $\{A,B\} \in S_k(G)$ as $(A,B)$ if and only if~$b \subseteq B$. 
A $k$-block $b$ in $G$ is \emph{separable} if it is the interior of some star in $S_k(G)$, i.e.\ if there exists a star $\sigma \subseteq \vS_k(G)$ such that $\interior(\sigma) = b$. 

Now suppose that $b$ is a separable $k$-block in $G$, and let $\tau$ be the $k$-profile induced by~$b$. Then the star~$\rho$ that consists precisely of the separations $(V(C) \cup N(C), V(G) \setminus V(C))$ for components~$C$ of~$G-b$ is in~$\vS_k(G)$ \cite[Lemma~4.1]{CG14:isolatingblocks}, and it is easy to see that $\interior(\rho) = b$ and that~$\tau$ is the unique $k$-profile which contains~$\rho$.
Moreover, we have $b \subseteq B$ for every $(A,B) \in \tau$, and thus $b \subseteq \interior(\sigma)$ for every star $\sigma \subseteq \tau$. Hence, if~$b$ is separable, then every star in~$\tau$ whose interior is of smallest size among all exclusive stars in~$\tau$ will be equal to~$b$. 
It follows by condition \ref{itm:GraphTanglesToTTDess} that the essential part of the \td\ $(T, \cV)$ from \cref{thm:GraphTanglesToTTDVersion} which is home to $\tau$ is equal to $b$. Thus, every separable $k$-block in~$G$ appears as a part of~$(T, \cV)$.  
All in all, we have the following theorem:

\begin{THM}\label{thm:TreeDecompWithNiceProperties}
    Let $G$ be a graph and $k \in \N$. Then there exists a \td\ $(T, \cV)$ of~$G$ of adhesion~$<k$ which has the following properties: 
    \begin{enumerate}[label=\rm{(\roman*)}]
        \item\label{itm:TreeDecompWithNicePropertiesEff} $(T, \cV)$ efficiently distinguishes all regular $k$-profiles in $G$;
        \item\label{itm:TreeDecompWithNicePropertiesSmallOrInduce} each part of $(T, \cV)$ of size $> 3k-3$ is home to a $k$-tangle, which it then witnesses and induces;
        \item\label{itm:TreeDecompWithNicePropertiesBlocks} every separable $k$-block in $G$ appears as a part of $(T, \cV)$. \qed
    \end{enumerate}
\end{THM}

\noindent We remark that the existence of \tds\ satisfying \ref{itm:TreeDecompWithNicePropertiesEff} and \cref{itm:TreeDecompWithNicePropertiesBlocks} is already known due to Carmesin and Gollin \cite{CG14:isolatingblocks} (see also \cite{SARefiningBlocks} for a short proof). What is new is property \ref{itm:TreeDecompWithNicePropertiesSmallOrInduce}: that all the inessential parts of the \td\ $(T, \cV)$ are small, while all of its parts that are not small induce $k$-tangles.

\section{Refining trees of tangles in abstract separation systems}\label{sec:RefiningEssentialStars}

In this section we extend \cref{thm:GraphTanglesToTTDVersion} to abstract separation systems. For this note that condition \ref{itm:GraphTanglesToTNSiness} of \cref{thm:GraphTanglesToTNestedSetVersion} is already formulated in a way which translates verbatim to abstract separation systems, while its condition \ref{itm:GraphTanglesToTNSess} makes use of the specific properties of graphs. More precisely, \ref{itm:GraphTanglesToTNSess} refers to the size of the interiors of the essential nodes, i.e.\ the number of vertices contained in them. Since vertices need not exist in the more general setting of abstract separation systems, we need to adapt this condition in order to generalize \cref{thm:GraphTanglesToTTDVersion}. 

For this, recall that the purpose of condition \ref{itm:GraphTanglesToTNSess} is to ensure that the refinement is optimal. More precisely, \ref{itm:GraphTanglesToTNSess} makes sure that the interiors of the essential nodes of the refined nested set are as small as possible and thereby guarantees that the essential nodes are as `close' to the tangle they are home to as possible. Here `small' refers to the number of vertices in the interiors. 
Since we cannot rely on the existence of vertices, as mentioned above, we need to find a different measure to determine the `closeness' of the essential nodes to the tangle they are home to. 
In particular, as the notion of an `interior' of a node does not extend to abstract separation systems, this measure needs to refer to the nodes themselves.

Since the nodes of a nested set are stars, we may consider the partial order on them.
For this,  recall that if $S$ is some separation system, then the partial order on $\vS$ induces a partial order on the set of all proper stars in $\vS$ (see \cref{subsec:SepSys}). 
Clearly, for a given profile $P$ of $S$, those stars in $P$ that are maximal among all stars in $P$ are closest to $P$.
We thus replace \ref{itm:GraphTanglesToTNSess} in \cref{thm:GraphTanglesToTNestedSetVersion} with the condition that every essential node is a maximal star in the tangle living at it.
We then obtain \cref{thm:abstracttanglesToT}, which is the main result of this section, and which we restate here. 
For this, set $\cT' := \{\{\vr, \vs\} \subseteq \vS : \vr \leq \sv \text{ and } \vr \vee \vs \text{ is co-small}\}$.

\begin{customthm}{\ref{thm:abstracttanglesToT}}
	\emph{Let $\vS$ be a submodular separation system inside some distributive universe, and let $\cF$ be a friendly set of stars in $\vS$ with $\cT' \subseteq \cF$. Further, let $\tilde{N} \subseteq S$ be a nested set that distinguishes all~$\cF$-tangles of $S$, and suppose that every separation in $\tilde{N}$ is good for the set of all $\cF$-tangles of $S$. Then there is a nested set $N \subseteq S$ with $\tilde{N} \subseteq N$ such that
	\begin{enumerate}[label=\rm{(\roman*)}]
		\item every inessential node of $N$ is an element of $\cF$;
		\item every essential node of $N$ is maximal in the $\cF$-tangle living at it. 
	\end{enumerate}}
\end{customthm}

\noindent Note that by \cite[Theorem~5.8]{SARefiningInessParts} there exists a nested set $\tilde{N}$ of separations which satisfies the premise of \cref{thm:abstracttanglesToT}. This set is even canonical, but in a slightly weaker sense than usual. 

Further, recall that `good' generalizes the notion of efficiency to separations systems that come without an order function (see \cref{sec:RefiningInessStars}).
\medskip

Before we prove \cref{thm:abstracttanglesToT}, let us first make some remarks about it.
The set $N$ refines the nodes of $\tilde{N}$ in that it `splits' them into smaller nodes which are each either a star in~$\cF$ if they are inessential, or `close' to the tangle they are home to.
More precisely,~\cref{itm:abstracttanglesToTiness} implies that every inessential node of the refinement is in fact too small to be home to a tangle, while~\cref{itm:abstracttanglesToTess} guarantees that the essential nodes are as `close' to the tangle they are home to as possible, where the `closeness' is measured by the partial order on the set of stars. 
In particular, condition \cref{itm:abstracttanglesToTess} ensures that the essential nodes of the refinement cannot be refined any further, and hence the refinement from \cref{thm:abstracttanglesToT} is optimal in that sense. 

An \emph{abstract tangle} of a separation system $S$ is a $\tilde{\cT}$-tangle where $\tilde{\cT} := \big\{\{\vr, \vs, \vt\} \subseteq \vS: \vr \vee \vs \vee \vt \text{ is co-small}\big\}$ \cite{AbstractTangles}.
Since $\tilde{\cT}^*$ (the set of all stars in $\tilde{\cT}$) is closed under shifting and $\tilde{\cT}^* \supseteq \cT'$, we can apply \cref{thm:abstracttanglesToT} to the set of all abstract tangles if $S$ lies inside some distributive universe.

Moreover, since every separation system of the form $S_k$ inside a submodular universe is submodular, \cref{thm:abstracttanglesToT} can also be applied to such systems. In fact, we obtain the following corollary by applying it to the nested set which Diestel, Hundertmark and Lemancyzk \cite[Theorem~3.6]{ProfilesNew} constructed:

\begin{COR}\label{cor:ToTSubmodularUniverse}
	Let $\vU$ be a distributive, submodular universe of separations, and let $k \in \N$. Further, let~$\cF$ be a friendly set of stars in $\vS_k$ such that $\cT' \subseteq \cF$, and let $\cP$ be the set of all $\cF$-tangles of $S_k$.
	Then there exist nested sets $\tilde{N} := \tilde{N}(\cP) \subseteq N \subseteq S_k$ such that
	\begin{enumerate}[label=\rm{(\roman*)}]
		\item\label{itm:ToTSubmodUni1} $\tilde{N}$ efficiently distinguishes all tangles in $\cP$;
		\item\label{itm:ToTSubmodUni2} for every automorphism $\alpha$ of $\vU$ we have $\tilde{N}(\cP^\alpha) = \tilde{N}(\cP)^\alpha$;
		\item\label{itm:ToTSubmodUni3} every inessential node in $N$ is a star in $\cF$;
		\item\label{itm:ToTSubmodUni4} every essential node in $N$ is maximal in the $\cF$-tangle living at it.
	\end{enumerate}
\end{COR}	
\begin{proof}
    By \cite[Theorem~3.6]{ProfilesNew} there exists a nested set $\tilde{N} \subseteq S_k$ which satisfies \cref{itm:ToTSubmodUni1} and \cref{itm:ToTSubmodUni2} such that every separation in $\tilde{N}$ distinguishes some pair of profiles in $\cP$ efficiently. Since every separation which efficiently distinguishes some pair of profiles in $\cP$ is good for $\cP$ by \cite[Proposition~3.4]{SARefiningInessParts}, applying \cref{thm:abstracttanglesToT} to $\tilde{N}$ yields the existence of a nested set $N \subseteq S_k$ with $\tilde{N} \subseteq N$ that satisfies \cref{itm:ToTSubmodUni3} and \cref{itm:ToTSubmodUni4}.
\end{proof}

Let us now turn to the proof of \cref{thm:abstracttanglesToT}. For this, let a submodular separation system $\vS$, a set $\cF$ of stars in $\vS$ and a nested set $\tilde{N} \subseteq S$ be given which satisfy the premise of \cref{thm:abstracttanglesToT}. 
By \cref{thm:RefiningGoodToTs}, there exists a refinement $N'$ of $\tilde{N}$ which satisfies \cref{itm:abstracttanglesToTiness}, i.e.\ which refines all the inessential nodes of $\tilde{N}$ so that all inessential nodes are stars in~$\cF$.
Therefore, in order to prove \cref{thm:abstracttanglesToT}, we only need to show that we can refine the essential nodes too, i.e.\ that we can `split' the essential nodes into smaller nodes so that all inessential nodes are stars in $\cF$, and the new essential nodes are maximal in the tangle they are home to.
For this, we prove the following variant of \cref{lem:ReflemEssStarsGraphs}:

\begin{LEM}\label{lem:ReflemForEssStars}
    Let $\vS$ be a submodular separation system inside a distributive universe $\vU$, and let $\cF \supseteq \cT'$ be a friendly set of stars in $\vS$. 
	Further, let $\sigma = \{\vs_1, ..., \vs_k\} \subseteq \vS$ be a star which is home to a unique~$\cF$-tangle~$P$ of $S$, and suppose that each $s_i$ is good for the set of all $\cF$-tangles of $S$. 
	Then there exist a maximal star $\sigma' \subseteq P$ and an $S$-tree over~$\cF' := \cF \cup \{\sigma'\} \cup \{\{\sv_1\}, ...,\{\sv_k\}\}$ in which each $\vs_i$ appears as a leaf separation. 
\end{LEM}

Just as \cref{thm:JERefining} and \cref{lem:ReflemEssStarsGraphs} imply \cref{thm:GraphTanglesToTTDVersion}, their generalizations to abstract separation systems (\cref{thm:RefiningGoodToTs} and \cref{lem:ReflemForEssStars}) imply \cref{thm:abstracttanglesToT}.

\begin{proof}[Proof of \cref{thm:abstracttanglesToT}]
    By applying \cref{thm:RefiningGoodToTs} to $\tilde{N}$ we obtain a refinement $N'$ of $\tilde{N}$ whose inessential nodes are all stars in $\cF$.
    Let $\cP$ be the set of all $\cF$-tangles of $S$, let $P \in \cP$ be given, and let $\sigma$ be the node of $N'$ that is home to $P$.
    By \cref{lem:ReflemForEssStars}, there exists a star $\sigma' \subseteq P$ with $\sigma \leq \sigma'$ that is maximal in $P$, and an $S$-tree over $\cF \cup \{\sigma'\} \cup \{\{\sv\} : \vs \in \sigma\}$ in which each $\vs \in \sigma$ appears as a leaf separation. Let $N_P$ be the set of separations induced by that $S$-tree. Clearly, $N_P$ is nested with~$N'$.
	Setting $N := N' \cup \bigcup_{P \in \cP} N_P$ then yields the claim.
\end{proof}

Let $\sigma \subseteq \vS$ be a star that satisfies the premise of \cref{lem:ReflemForEssStars}, and let $P$ be the unique $\cF$-tangle living at $\sigma$. One particularly natural approach to construct an $S$-tree as in \cref{lem:ReflemForEssStars} is the following.
Take a maximal star $\sigma' \subseteq P$ with $\sigma \leq \sigma'$, and consider the set~$N'_P := \{s : \vs \in \sigma \cup \sigma'\}$. If~$\sigma$ is not already maximal in~$P$, then $\sigma < \sigma'$, which implies that $N'_P$ has some inessential nodes.
As these nodes will in general not be stars in $\cF$, we need to refine them further with $S$-trees as in \cref{lem:reflem}.
If that is possible, i.e.\ if there exists for every inessential node~$\rho$ of~$N'_P$ an $S$-tree over $\cF \cup \{\{\rv\} : \vr \in \rho\}$, then these $S$-trees give rise to one overall $S$-tree as in \cref{lem:ReflemForEssStars}.

However, this seemingly easy approach entails some difficulties. \cref{lem:reflem} relies on the assumption that the inverse of every separation in the inessential star which we want to refine is closely related to some $\cF$-tangle of $S$. Since this assumption is necessary, as pointed out earlier, we cannot pick the star~$\sigma'$ arbitrarily among the maximal stars in $P$. Indeed, as the inverse of every separation in $\sigma' \setminus \sigma$ will be contained in an inessential node of $N'_P$, we have to pick a star $\sigma' \subseteq P$ which is closely related to $P$ in order to be able to refine the newly arising inessential nodes.
Unfortunately though, we do not know whether there always exists a star which is maximal in $P$ and closely related to~$P$. 

Still, we hold on to the strategy of first choosing a suitable star $\sigma'$ in $P$ with $\sigma \leq \sigma'$, and then refining the inessential nodes of $N'_P$.
To make this strategy work, we relax the conditions that we impose on the star~$\sigma'$. 
Instead of requiring $\sigma'$ to be a maximal star in~$P$, we only ask $\sigma'$ to be `near-maximal' in~$P$, a slightly weaker property that will allow us to find a star $\sigma'$ which is `near-maximal' in and closely related to~$P$. Let us first make this property formal.

We call a set $R \subseteq P$ \emph{narrow in $P$} if for every $\vx \in P$ the separation $\xv \vee \bigvee_{\vr \in R} \vr$ is co-small.
Further, a star~$\sigma \subseteq P$ is \emph{near-maximal in~$P$} if $\sigma$ is narrow in $P$ and for every~$\vx \in P$ there is at most one $\vs \in \sigma$ such that~$\vs \leq \vx$.

The main part of the proof of \cref{lem:ReflemForEssStars} is to show that we can always find a star $\sigma' \geq \sigma$ which is both closely related to and near-maximal in $P$.

\begin{LEM}\label{lem:NearMaxAndClRelStar}
    Let $\vS$ be a submodular separation system inside some distributive universe, and let $P$ be a profile of $S$. Further, let $\sigma \subseteq P$ be a star which is good for the set of all profiles of $S$. 
	Then there exists a star $\sigma' \subseteq P$ with $\sigma \leq \sigma'$ which is closely related to and near-maximal in~$P$.
\end{LEM}

With \cref{lem:NearMaxAndClRelStar} at hand, we can show \cref{lem:ReflemForEssStars}:

\begin{proof}[Proof of \cref{lem:ReflemForEssStars}]
    By applying \cref{lem:NearMaxAndClRelStar} to $\sigma$ we obtain a star $\sigma' \subseteq P$ with $\sigma \leq \sigma'$ which is both closely related to and near-maximal in $P$. Set $N' := \{s : \vs \in \sigma \cup \sigma'\}$, and let $\Sigma$ be the set of inessential nodes of $N'$. By assumption, every inessential node $\rho \in \Sigma$ contains only separations whose inverse is closely related to some $\cF$-tangle of $S$.
    Indeed, since $\sigma'$ is home to $P$ and hence essential, we have $\rho \neq \sigma'$. In particular, $\rho$ contains no separation from $\sigma'$. Thus, by the assumptions on $\sigma$ and $\sigma'$, the star $\rho$ satisfies the premise of \cref{lem:reflem}. So we may apply
    \cref{lem:reflem} to $\rho$ to obtain an $S$-tree over $\cF \cup \{\{\rv\} : \vr \in \rho\}$ in which each $\vr \in \rho$ appears as a leaf separation. Let $N_\rho$ be the set of separations induced by that $S$-tree. Clearly, $N_\rho$ is nested with~$N'$. 
    
    Take a maximal star $\sigma'' \subseteq P$ with $\sigma' \leq \sigma''$. By the definition of near-maximal, every inessential node $\rho$ of $\{s : \vs \in \sigma' \cup \sigma''\}$ has one of the following two forms. Either $\rho = \{\sv\}$ for a separation $\vs \in \sigma''$, which then must be small, or $\rho = \{\vr, \sv\}$ where $\vr \in \sigma'$ and $\vs \in \sigma''$ such that $\vr \vee \sv$ is co-small. In the first case, we have $\rho \in \cF$ because $\cF$ is friendly, and in the second case we have $\rho \in \cT' \subseteq \cF$.
    
    Thus, $\tilde{N} := N' \cup \bigcup \{N_\rho : \rho \in \Sigma\} \cup \{s : \vs \in \sigma''\}$ is a nested set of separations that induces an $S$-tree over $\cF \cup \{\sigma''\} \cup \{\{\sv\} : \vs \in \sigma\}$ in which each $\vs \in \sigma$ appears as a leaf separation.
\end{proof}

The remainder of this section is devoted to the proof of \cref{lem:NearMaxAndClRelStar}. For this, for the remainder of this section, let $S$ be a submodular separation system. Further, let $P$ be a profile of $S$, and let~$\sigma \subseteq P$ be a \emph{star of good separations}, that is a star of separations which are all good for the set of all profiles of~$S$. 
We define $P_\sigma := \{\vs \in P : \vs \text{ is nested with } \sigma\}$ to be the set of all separations in~$P$ that are nested with~$\sigma$.

The proof of \cref{lem:NearMaxAndClRelStar} consists of two steps. First, we show that there exists a star in~$P_\sigma$ which is both closely related to and narrow in $P$. 
Then, we prove that every star in $P_\sigma$ which contains the least number of separations such that it is closely related to and narrow in $P$ is near-maximal in~$P$. This then concludes the proof of \cref{lem:NearMaxAndClRelStar}.

To find a star $\sigma' \subseteq P_\sigma$ which is both closely related to and narrow in $P$, we consider the set $R^P_\sigma$ of all maximal separations in $P_\sigma$. By \cref{prop:MaxSepsAreCloselyRelated}, $R^P_\sigma$ is closely related to~$P$. Moreover, it is easy to check, as we will do later, that $R^P_\sigma$ is narrow in $P$. However, $R^P_\sigma$ is in general not nested and hence not a star. Therefore, in the following, we introduce a method\footnote{This generalizes a method described in \cite[Ch.\ 4.1]{JoshRefining}.} of turning $R^P_\sigma$ into a star which preserves the property of $R^P_\sigma$ of being closely related to and narrow in $P$. 

Let $\vr, \vs \in \vS$ be two crossing separations. An \emph{unscrambling} of $\vr, \vs$ are two separations~$\vr', \vs' \in \vS$ such that
\[
\vr \wedge \sv \leq \vr' \leq \vr \text{ and } \vs \wedge \rv \leq \vs' \leq \vs \text{ and at least one of } \vs' = \vs \wedge \rv' \text{ and } \vr' = \vr \wedge \sv' \text{ holds}
\]

\begin{figure}[h]
\centering
\scalebox{0.8}{%
\begin{tikzpicture}
\draw[thick] (0,0) -- (6,3);
\draw[thick] (0,3) -- (6,0);

\draw[blue, thick] (0, 2.75) -- (2.9, 1.3) -- (3.5, 0);
\draw[blue, thick] (3.7, 0) -- (3.1, 1.3) -- (6, 2.75);

\draw[black, thick, dashed] (0, 2.5) -- (2.25, 1.35) -- (0, 0.25);
\draw[black, thick, dashed] (6, 0.25) -- (3.75, 1.35) -- (6, 2.5);

\draw[->, thick] (1.5, 2.25) -- (1.75, 2.75) node[anchor=south east] {$\vr$}; 
\draw[->, thick] (4.5, 2.25) -- (4.25, 2.75) node[anchor=south west] {$\vs$};
\draw[->, thick, blue] (2, 1.75) -- (2.375, 2.5) node[anchor=south] {$\vr'$};
\draw[->, thick, blue] (4, 1.75) -- (3.625, 2.5) node[anchor=south] {$\vs'$};
\end{tikzpicture}}
\caption{An unscrambling $\vr', \vs'$ of two crossing separations $\vr$ and $\vs$.}
\label{fig:DefUnscrambling}
\end{figure}

\noindent (see \cref{fig:DefUnscrambling}). Note that we do not require $\vr' = \vr \wedge \sv'$ \emph{and} $\vs' = \vs \wedge \rv'$ because in our applications we want to choose a suitable separation $\vr'$ and then define $\vs' := \vs \wedge \rv'$. But then, in general, $\vr \wedge \sv'$ need not be equal to $\vr'$. However, if~$U$ is distributive, then both equations hold.

\begin{proposition}\label{prop:DefOfUnscramblingInDistrUni}
	Let $\vr, \vs$ be two crossing separations in a distributive universe $\vU$, and let $\vr', \vs'$ be any unscrambling of them. Then $\vr' = \vr \wedge \sv'$ and $\vs' = \vs \wedge \rv'$. 
\end{proposition}
\begin{proof}
    Since $\vr', \vs'$ is an unscrambling of $\vr, \vs$, at least one equation has to hold by definition; by symmetry, we may assume $\vs' = \vs \wedge \rv'$. 
    It follows that
    \[
    \vr \wedge \sv' = \vr \wedge (\vs \wedge \rv')^* = \vr \wedge (\sv \vee \vr') = (\vr \wedge \sv) \vee (\vr \wedge \vr') = (\vr \wedge \sv) \vee \vr' = \vr'
    \]
    since $\vr \wedge \sv \leq \vr' \leq \vr$ by definition.
\end{proof}

If $R$ is a set of oriented separations, we call $R'$ an \emph{unscrambling of $R$} if $R'$ can be obtained from $R$ by successively unscrambling the crossing separations in $R$. 
We remark that the unscrambling of two separations is in general not unique. Moreover, the unscrambling of a set of more than two separations might further depend on the order in which the separations were unscrambled.
\medskip

We now begin with the proof of \cref{lem:NearMaxAndClRelStar}. First, we show that for every pair of crossing separations $\vr, \vs \in P$ which are closely related to~$P$ there exists an unscrambling of them which is closely related to~$P$. In addition, we prove that if $r, s$ are nested with~$\sigma$, then this unscrambling can be chosen to be nested with~$\sigma$. As the set~$R_\sigma^P$ of all the maximal separations in~$P_\sigma$ is closely related to~$P$ by \cref{prop:MaxSepsAreCloselyRelated}, this allows us to turn~$R_\sigma^P$ into a star $\sigma'$ which is still closely related to~$P$ and nested with~$\sigma$, by successively unscrambling the crossing separations in~$R^P_\sigma$. 

We then show that this star $\sigma'$ is not just closely related to $P$ and nested with $\sigma$ but in fact also narrow in $P$. For this, we prove that $R_P^\sigma$ is narrow in $P$, and that every unscrambling of a set of separations which is narrow in $P$ is still narrow in $P$ if $S$ is contained in a distributive universe. We remark that this is in fact the only point where we need the assumption of \cref{lem:NearMaxAndClRelStar} that the universe in which $S$ lies is distributive. 

Lastly, we show that every star in $P_\sigma$ which contains the least number of elements among all stars in $P_\sigma$ that are closely related to and narrow in~$P$ is near-maximal in $P$.

\begin{LEM}\label{lem:unscrambling}
	Let $\vS$ be a submodular separation system, and let $\vr, \vs \in \vS$ be two crossing separations which are closely related to some profile $P$ of $S$. Further, let $\sigma \subseteq P$ be a star of good separations that is nested with $r$ and $s$, and let $\vr' \in P_\sigma$ be a separation which is minimal under all separations in $P_\sigma$ that are closely related to $P$ and greater or equal to $\vr \wedge \sv$. Then $\vs \wedge \rv'$ is an element of $\vS$ and closely related to $P$. 
\end{LEM} 

\noindent By definition, $\vr', \vs \wedge \rv'$ is an unscrambling of $\vr, \vs$, and it is nested with $\sigma$ by \cref{lem:Fischlemma}. Since $\vr$ is a candidate for $\vr'$, \cref{lem:unscrambling} implies that there exists an unscrambling of $\vr, \vs$ which is closely related to $P$ and nested with $\sigma$.

Note that for separation systems of the form $S_k$ inside some submodular universe, Erde showed a related statement which, similar as to \cref{lem:unscrambling}, implies the existence of unscramblings that preserve the property of being closely related to some profile, but without the additional assertion of \cref{lem:unscrambling} that the unscrambling can be chosen to be nested with~$\sigma$ \cite[Lemma~4.3]{JoshRefining}. There, the choice of the unscrambling relied on the existence of an order function.

\begin{proof}[Proof of \cref{lem:unscrambling}]
	Let $\vs' \in P_\sigma$ be a separation which is minimal among all separations in $P_\sigma$ that are closely related to~$P$ and greater or equal to $\vs \wedge \rv'$. For this, observe that~$\vs$ is a candidate for~$\vs'$.
	If $\vs' = \vs \wedge \rv'$, we are done, so suppose that~$\vs \wedge \rv' < \vs'$. 
	By the submodularity of~$\vS$, at least one of $\vr' \wedge \sv'$ and $\rv' \wedge \vs'$ is an element of~$\vS$; by symmetry\footnote{Note that the situation itself is not completely symmetric in $r'$ and $s'$ since we chose $s'$ depending on $r'$. However, if~$\vr' \wedge \sv' \in \vS$, then swapping $r'$ and $s'$ in the remainder of this proof does not change anything.} we may assume that $\vs'' := \rv' \wedge \vs' \in \vS$. Note that $\vs'' = \rv' \wedge \vs' \leq \rv' \wedge \vs = \rv' \wedge (\rv' \wedge \vs) \leq \rv' \wedge \vs'$, and thus $\vs'' = \rv' \wedge \vs$. 
	
	By the choice of $\vs'$, we have that $\vs''$ is not closely related to $P$, and so, by \cref{prop:ShowingThatASepIsCloselyRelated}, there is a separation $\vt \leq \vs' \in P$ such that $\vs'' \wedge \vt \notin \vS$. We take a separation $\vt \in \vS$ which is minimal with that property. 
	Since $\rv' \wedge \vt = \rv' \wedge (\vs' \wedge \vt) = (\rv' \wedge \vs') \wedge \vt = \vs'' \wedge \vt \notin \vS$ by assumption, it follows, by the submodularity of~$\vS$, that $\vr'' := \vr' \wedge \tv \in \vS$. Note that $\sv \leq \tv$ by the choice of $\vt$ and $\vr \wedge \sv \leq \vr'$ by assumption, and thus $\vr \wedge \sv \leq \vr''$.

    Since $\sigma$ is good by assumption, we can fix, for every $\vy \in \sigma$, some profile $Q_y$ of $S$ which has the property that $\yv$ is closely related to it. We then set
	\[
	\rho := \{\vy \in \sigma : \rv'' \in Q_y\}.
	\]
	Now first assume that $\vr''$ is closely related to $P$ and nested with $\rho$. Then the assertion follows: By assumption, the star $\sigma' := \{\vy \in \sigma : y \text{ crosses } r''\}$ is a subset of $\sigma \setminus \rho$, and hence \cref{prop:InfimaOfClRelSetsAreClRel} yields that the separation $\vz := \vr'' \wedge \bigwedge_{\vy \in \sigma'} \yv$ is contained in $\vS$ and closely related to $P$. 
	Moreover, as $\vr \wedge \sv \leq \vr''$ and as $\vr \wedge \sv$ is nested with $\sigma$ by assumption and \cref{lem:Fischlemma}, we have $\vr \wedge \sv \leq \vz$.
	But since $\vz \leq \vr'' < \vr'$, this contradicts the minimal choice of $\vr'$.
	
	Therefore, it suffices to show that $\vr''$ is closely related to~$P$ and nested with $\rho$. We first show the former. For this, suppose for a contradiction that $\vr''$ is not closely related to~$P$. By \cref{prop:ShowingThatASepIsCloselyRelated}, this is witnessed by a separation $\vt' \leq \vr'$ in that $\vr'' \wedge \vt' \notin \vS$. As $\vr'' \wedge \vt' = (\vr' \wedge \tv) \wedge \vt' = \tv \wedge (\vr' \wedge \vt') = \tv \wedge \vt'$, it follows, by the submodularity of $\vS$, that $\vt \wedge \tv' \in \vS$. 
	But this contradicts the minimal choice of $\vt$: We have $\vt \wedge \tv' < \vt$ but still $\vs'' \wedge (\vt \wedge \tv') = (\vs'' \wedge \tv') \wedge \vt = \vs'' \wedge \vt \notin \vS$ as $\vs'' \leq \rv' \leq \tv'$. Therefore, $\vr''$ is closely related to $P$.
	
	To conclude the proof, we are left to show that $\vr''$ is nested with $\rho$. As $\vr'' = \vr' \wedge \tv$, and~$r'$ is nested with~$\sigma \supseteq \rho$ by assumption, it suffices, by \cref{lem:Fischlemma}, to show that $t$ is nested with~$\rho$. For this, suppose for a contradiction that $t$ crosses a separation $\vy \in \rho$. 
    Then $\vs'' \leq \yv$. Indeed, since $\vs'' \leq \rv''$ and $y$ crosses $r''$ but not $s''$, we can only have $\vs'' \leq \yv$ or $\vs'' \leq \vy$; but the latter case is not possible as $\vs'' \not\leq \vx$ for all $\vx \in \sigma$ by construction.
    Further, as $\vt \leq \rv'' \in Q_y$ and $\yv$ is closely related to $Q_y$, it follows that $\vt \wedge \yv \in \vS$. 
    But this again contradicts the minimal choice of $\vt$: We have $\vt \wedge \yv < \vt$ and still $\vs'' \wedge (\vt \wedge \yv) = (\vs'' \wedge \yv) \wedge \vt = \vs'' \wedge \vt \notin \vS$.
\end{proof}

We can now conclude that for every set of separations which are closely related to some profile there is an unscrambling which is again closely related to that profile.

\begin{COR}\label{cor:unscramblingsets}
	Let $R \subseteq \vS$ be a set which is closely related to some profile~$P$ of~$S$. Further, let $\sigma \subseteq P$ be a star of good separations and suppose that $R$ is nested with $\sigma$. Then there is an unscrambling $R'$ of $R$ that is closely related to~$P$ and nested with $\sigma$.
	Moreover, if $R$ contains for some $\vs \in \sigma$ a separation $\vr$ with~$\vs \leq \vr$, then $R'$ does so as well.
\end{COR}

\begin{proof}
    Let $R = \{\vr_1, ..., \vr_n\}$ be some enumeration of $R$. 
    We unscramble the separations in~$R$ inductively by repeatedly applying \cref{lem:unscrambling} as follows. 
    Set $R^0 := R$, and assume that we already constructed a set $R^{k-1} := \{\vr^{k-1}_1, \ldots, \vr^{k-1}_n\}$ which is closely related to $P$ and nested with $\sigma$. We then pick a pair of crossing separations $r_i^{k-1}, r_j^{k-1}$ and apply \cref{lem:unscrambling} to them, yielding an unscrambling $\vr^k_i, \vr^k_j$ of $\vr^{k-1}_i, \vr^{k-1}_j$ which is closely related to $P$ and nested with $\sigma$. 
    In every step, we thus replace two crossing separations $\vr^{k-1}_i, \vr^{k-1}_j$ with two nested separations $\vr^k_i \leq \vr^{k-1}_i$ and $\vr^k_j \leq \vr^{k-1}_j$ such that $\vr^k_i \leq \rv^k_j$, so we do not need to consider a pair of indices in $[n]$ twice. Hence, after at most $N := \frac{1}{2}n(n-1)$ steps, this process terminates, and we obtain a set~$R' := R^N$ that is an unscrambling of $R$, closely related to $P$ and nested with $\sigma$.

    For the `moreover'-part, suppose that there are separations $\vs \in \sigma$ and $\vr \in R$ such that $\vs \leq \vr$. We show by induction on $k \in [N]$ that there is a separation $\vr^k_i \in R^k$ with $\vs \leq \vr^k_i$, which clearly implies the assertion for $k := N$. The base case $k = 0$ holds by assumption, so we may assume that $k > 0$ and that there exists a separation $\vr^{k-1}_i \in R^{k-1}$ with $\vs \leq \vr^{k-1}_i$. If $\vr^{k-1}_i = \vr^k_i$, we are done, so we may assume that we unscrambled~$\vr^{k-1}_i$ and some crossing separation~$\vr^{k-1}_j$ in the $k$-th step. 
    In particular, this means that~$\vr^{k-1}_j \not\leq \vs$ since otherwise $\vr^{k-1}_j \leq \vs \leq \vr^{k-1}_i$ would contradict that $r^{k-1}_j$ and $r^{k-1}_i$ cross. 
    Moreover, as~$P$ is consistent and $\vr^{k-1}_j, \vs \in P$, we cannot have $\rv^{k-1}_j \leq \vs$. 
    Therefore, since $r^{k-1}_j$ and~$s$ are nested by construction, either $\vs \leq \rv^{k-1}_j$ or $\vs \leq \vr^{k-1}_j$. In the first case we have $\vs \leq \vr^{k-1}_i \wedge \rv^{k-1}_j \leq \vr^k_i$, while in the second case we either have $\vs \leq \vr^{k}_i$ or $\vs \leq \vr^k_j$ by the definition of an unscrambling.
\end{proof}

Next, we show that the set $R^P_\sigma$ of all maximal separations in $P_\sigma$ is narrow in $P$.

\begin{proposition}\label{prop:TheSetOfMaxSepsInPsigmaIsNarrowInP}
    Let $P$ be a profile of $S$, and let $\sigma \subseteq P$ be a star of good separations. Further, let $R^P_\sigma \subseteq P_\sigma$ be the set of all maximal separations in $P_\sigma$. Then $R^P_\sigma$ is narrow in $P$.
\end{proposition}

\begin{proof}
    Let $\vx \in P$ be given; we need to show that $\xv \vee \bigvee_{R_\sigma^P} \vr$ is co-small. For this, let $\vx' \in P$ with $\vx \leq \vx'$ be maximal in $P$, and fix for every separation $\vs \in \sigma$ a profile $Q_s$ of $S$ to which~$\sv$ is closely related. Then $\vx' \in Q_s$ for every $\vs \in \sigma_{x'} := \{\vs \in \sigma : \vs \text{ crosses } \vx'\}$ since otherwise $\vs \vee \vx' = (\sv \wedge \xv') \in \vS$, and thus $\vs \vee \vx' \in P$ as $P$ is a profile, which would contradict the maximality of $\vx'$ in $P$.
    It follows, by \cref{prop:InfimaOfClRelSetsAreClRel}, that $\vy := \vx' \wedge \bigwedge_{\vs \in \sigma_{x'}} \sv$ is an element of~$\vS$. By construction,~$y$ is nested with~$\sigma$, and thus~$P_\sigma$ contains an orientation of~$y$; by consistency, we have $\vy \in P_\sigma$. This implies that there is a separation $\vy' \in R^P_\sigma$ with~$\vy \leq \vy'$.
    Thus,
    \begin{align*}
    \xv \vee \bigvee_{\vr \in R^P_\sigma} \vr &\geq \xv' \vee \bigvee_{\vr \in R^P_\sigma} \vr = \xv' \vee \vy' \vee \bigvee_{\vr \in R^P_\sigma} \vr \geq \xv' \vee \Big(\vx' \wedge \bigwedge_{\vs \in \sigma_{x'}} \sv\Big) \vee \bigvee_{\vr \in R^P_\sigma} \vr \\
    &\geq \xv' \vee \Big(\vx' \wedge \bigwedge_{\vs \in \sigma_{x'}} \sv\Big) \vee \bigvee_{\vs \in \sigma_{x'}} \vs = \Big(\vx' \wedge \bigwedge_{\vs \in \sigma_{x'}} \sv\Big)^* \vee \Big(\vx' \wedge \bigwedge_{\vs \in \sigma_{x'}} \sv\Big)
    \end{align*}
    is co-small where the second last inequality holds since for every $\vs \in \sigma_{x'}$ there is some $\vr \in R^P_\sigma$ with~$\vs \leq \vr$.
\end{proof}

The next lemma shows that every unscrambling of $R^P_\sigma$ is narrow in $P$.

\begin{LEM}\label{lem:NarrowStaysNarrow}
    Let $\vS$ be a submodular separation system inside some distributive universe, and let $P$ be a profile of~$S$. Further, let $R \subseteq P$, and let $R'$ be any unscrambling of $R$. If $R$ is narrow in~$P$, then $R'$ is narrow in $P$.
\end{LEM}

\begin{proof}
    Since $R'$ is an unscrambling of $R$, there is a sequence $R = R_0, R_1, \ldots, R_n = R'$ of sets $R_i \subseteq \vS$ where~$R_{i+1}$ is obtained from $R_{i}$ by replacing two crossing separations in $R_i$ with an unscrambling of them.
	
	We show by induction on $i \in \{0, \ldots n\}$ that $\xv \vee \bigvee_{\vs \in R_i} \vs$ is co-small for every $\vx \in P$. For $R_n = R'$ this yields the claim. 
	The base case is clear since $R_0 := R$ is narrow in $P$ by assumption, so let~$i \geq 0$, and assume that $R_i$ is narrow in $P$. Further, let $\vr, \vs \in R_{i}$ be the separations that are unscrambled in the~$(i+1)$-st step. 
	Recall that by \cref{prop:DefOfUnscramblingInDistrUni} the unscrambling of two separations~$\vr, \vs$ in a distributive universe yields two separations $\vr' = \vr \wedge \sv'$ and $\vs' = \vs \wedge \rv'$.
	We need to show that for every $\vx \in P$ the separation
	\begin{equation*}
		\xv \vee \vr' \vee (\vs \wedge \rv') \vee \bigvee_{\vt \in R_{i} \setminus \{\vr, \vs\}} \vt = \Big(\xv \vee \vr' \vee \vs \vee \bigvee_{\vt \in R_{i} \setminus \{\vr, \vs\}} \vt\Big) \wedge \Big(\xv \vee \vr' \vee \rv' \vee \bigvee_{\vt \in R_{i} \setminus \{\vr, \vs\}} \vt\Big) =: \vy \wedge \vz
	\end{equation*}
	is co-small. For this, observe that for every pair of separations $\vu, \vw \in \vS$ we have that $\vu \wedge \vw$ is co-small whenever $\vu$ and $\vw$ are co-small and $\vu \geq \wv$ and $\vw \geq \uv$ since then $\vu \wedge \vw \geq \uv, \wv$ and thus $\vu \wedge \vw \geq \uv \vee \wv$.
	
	Now $\vz$ is clearly co-small since $\vr' \vee \rv'$ is. Further, it holds that $\vy \geq \vr' \geq \zv$ and $\vz \geq \rv' \geq \yv$.
	Therefore, we are left to show that $\vy$ is co-small. For this, we rearrange $\vy$ to
	\[
	\vy = \Big(\xv \vee (\vr \wedge \sv') \vee \vs \vee \bigvee_{\vt \in R_{i} \setminus \{\vr, \vs\}} \vt\Big) 
	= \Big(\xv \vee \vr \vee \vs \vee \bigvee_{\vt \in R_{i} \setminus \{\vr, \vs\}} \vt\Big) \wedge \Big(\xv \vee \sv' \vee \vs \vee \bigvee_{\vt \in R_{i} \setminus \{\vr, \vs\}} \vt\Big) =: \va \wedge \vb,
	\]
	where $\va$ is co-small by the induction hypothesis and $\vb$ is co-small since $\sv' \vee \vs$ is. 
	Further, $\va \geq \vs \geq \vs' \geq \bv$ and $\vb \geq \sv' \geq \sv \geq \av$, and hence $\vy = \va \wedge \vb$ is co-small. Therefore, $R'$ is narrow in $P$.
\end{proof}

We are now ready to prove \cref{lem:NearMaxAndClRelStar}:

\begin{proof}[Proof of \cref{lem:NearMaxAndClRelStar}]
    By \cref{prop:MaxSepsAreCloselyRelated} and \ref{prop:TheSetOfMaxSepsInPsigmaIsNarrowInP}, $R^P_\sigma$ is closely related to and narrow in~$P$. Applying \cref{cor:unscramblingsets} to $R^P_\sigma$ thus yields an unscrambling $R'$ of $R^P_\sigma$ which is closely related to $P$ and nested with $\sigma$. Moreover, by \cref{lem:NarrowStaysNarrow}, $R'$ is narrow in $P$. Then $\rho := \{\vr \in R' : \vr \not< \vs \text{ for all } \vs \in R'\}$ is a star which is closely related to $P$, and it is clear from the definition that $\rho$ is also narrow in $P$. Further, we have $\sigma \leq \rho$ by the `moreover'-part of \cref{cor:unscramblingsets}. 
	Thus, we can take a star $\sigma' \subseteq P$ such that:
 
	\begin{enumerate}[label=(\arabic*)]
		\item\label{itm:quasimaxstars1} $\sigma \leq \sigma'$ and $\sigma'$ is closely related to and narrow in $P$;
		\item\label{itm:quasimaxstars2} the number of elements in $\sigma'$ is minimal under all stars in $P$ that satisfy \ref{itm:quasimaxstars1}. 
	\end{enumerate}
    
    We claim that $\sigma'$ is near-maximal in $P$, which clearly implies the assertion.
    For this, suppose for a contradiction that $\sigma'$ is not near-maximal in $P$, i.e.\ there is a separation~$\vz \in P$ which is greater than at least two separations $\vr, \vr' \in \sigma'$. 
    We first show that we may assume $\vz \in P_\sigma$. Indeed, let~$\vy$ be any maximal separation in $P$ with $\vz \leq \vy$. By the assumption on~$\sigma$, we may fix for every separation $\vs \in \sigma$ a profile~$Q_s$ of~$S$ such that~$\sv$ is closely related to~$Q_s$. If $\yv \in Q_s$ for a some $\vs \in \sigma_y := \{\vs \in \sigma : s \text{ crosses } y\}$, then $\vy \vee \vs = (\yv \wedge \sv)^*$ is an element of~$\vS$ and thus contained in~$P$ because~$P$ is a profile, which contradicts the maximality of~$\vy$ in~$P$. Thus, $\vy \in Q_s$ for all $\vs \in \sigma_y$. Then \cref{prop:InfimaOfClRelSetsAreClRel} implies that $\vy' := \vy \wedge \bigwedge_{\vs \in \sigma_y} \sv$ is an element of $\vS$. By \cref{lem:Fischlemma}, $y'$ is nested with $\sigma$, and thus $\vy' \in P_\sigma$. Moreover, since $\sigma$ is a star, we still have $\vr, \vr' \leq \vy'$, and thus $\vy'$ is greater than at least two separations in~$\sigma$.

    Therefore, there is a separation in $P_\sigma$ which is greater than at least two separations in~$\sigma$, and we let~$\vx \in P_\sigma$ be any maximal such separation. Set $R := \{\vs \in \sigma' : \vs \not\leq \vx\} \cup \{\vx\}$. By construction, we have $\bigvee_{\sigma'} \vs \leq \bigvee_{R} \vs$, which implies that $R$ is narrow in $P$. Moreover,~$\vx$ is closely related to $P$ by \cref{prop:MaxSepsAreCloselyRelated}.
	Applying \cref{cor:unscramblingsets} to $R$ yields a star $\sigma''$ which is closely related to~$P$ and narrow in~$P$ by \cref{lem:NarrowStaysNarrow}. Moreover, $\sigma \leq \sigma''$ by the `moreover'-part of \cref{cor:unscramblingsets}, and thus $\sigma''$ satisfies~\ref{itm:quasimaxstars1}. But since $\vx$ was greater than at least two separations in $\sigma'$, it follows that $\abs{\sigma''} = \abs{R} < \abs{\sigma'}$, which contradicts that $\sigma'$ satisfies~\ref{itm:quasimaxstars2}. Hence, $\sigma'$ is near-maximal in $P$.
\end{proof}

As mentioned earlier, we only needed the assumption that $S$ lies inside a \emph{distributive} universe in order to prove that the unscrambling of a set which is narrow in some profile is still narrow in that profile. 
Therefore, along the lines of the proof of \cref{thm:abstracttanglesToT}, one can show that even if $S$ lies inside some non-distributive universe, it is still possible to refine any nested set $\tilde{N} \subseteq S$ that satisfies the premise of \cref{thm:abstracttanglesToT}. The refinement~$N$ of~$\tilde{N}$ then has the property that all its inessential nodes are stars in $\cF$, and all its essential nodes are unscramblings of the maximal separations in the tangle living at that node.

As the essential nodes of $N$ are such unscramblings, they should still be intuitively `close' to the tangle they are home to. 
However, since the assumption that the universe is distributive was crucial for the proof of \cref{lem:NarrowStaysNarrow}, we cannot say much about the `closeness' in terms of the partial order on $\vS$. 
It thus remains an open problem whether \cref{lem:ReflemForEssStars}, and hence \cref{thm:abstracttanglesToT}, still hold in separation systems inside universes that are not distributive.

\bibliographystyle{plain}
\bibliography{collective.bib}

\end{document}